\documentclass[12pt, twoside]{article}
\usepackage{amsmath,amsthm,amssymb}
\usepackage{times}
\usepackage{enumerate}
\usepackage{color}
\usepackage{comment}
\usepackage{graphicx}
\usepackage{subcaption}
\theoremstyle{definition}
\newtheorem{thm}{Theorem}[section]

\newtheorem{lem}[thm]{Lemma}
\newtheorem{defin}[thm]{Definition}

\numberwithin{equation}{section}
\begin{document}


\baselineskip=17pt


\title{A free boundary problem describing migration into rubbers -- quest of the large time behavior}

\author{Kota Kumazaki\\
{\small Faculty of Education, Nagasaki University,}\\
{\small 1-14 Bunkyo-cho, Nagasaki-city, Nagasaki, 851-8521, Japan}\\
{\small k.kumazaki@nagasaki-u.ac.jp}\\
Toyohiko Aiki\\
{\small Department of Mathematics Faculty of Sciences, Japan Women's University,}\\
{\small 2-8-1, Mejirodai, Bunkyo-ku, Tokyo, 112-8681, Japan} \\
{\small aikit@fc.jwc.ac.jp}\\
Adrian Muntean\\
{\small Department of Mathematics and Computer Science, Karlstad University,}\\
{\small Universitetsgatan 2, 651 88 Karlstad, Sweden}\\
{\small adrian.muntean@kau.se}
}

\date{}

\maketitle

\begin{abstract}
In many industrial applications, rubber-based materials are routinely used in conjunction with various penetrants or diluents in gaseous or liquid form. It is of interest to estimate theoretically the penetration depth as well as  the amount of diffusants stored inside the material. In this framework, we prove the global solvability and explore the large time-behavior of
solutions to a one-phase free boundary problem with nonlinear kinetic condition that is able to describe the migration of diffusants into rubber. 
The key idea in the proof of the large time behavior is to benefit of  a contradiction argument, since it is difficult to obtain uniform estimates for the growth rate of the free boundary due to the use of a Robin boundary condition posed at the fixed boundary.   
\end{abstract}

\section{Introduction}

\label{intro}

In many industrial applications, the behavior of rubber-based materials is difficult to predict theoretically. This intriguing fact is  especially due to their internal structure which allows for unexpected local changes (deformations, concentration localization, network entaglement, etc.) typically facilitated by the absorption and migration  of diffusants into the material; this is from where our motivation stems. There is a variety of possible modeling approaches for such scenarios. Motivated by our recent work \cite{NMKAMWG}, where solutions to our free boundary model did recover experimental data, we choose to follow a macroscopic modeling approach with kinetically-driven interfaces capturing the penetration of diffusants into the material. We refer the reader, for instance, to \cite{Fasano1,Fasano2,Friedman,Edwards} for closely related work especially what concerns the mathematics of Case II diffusion as it arises for some classes of polymers, but not directly applicable to the rubber case.
 
In this paper we consider the mathematical analysis of the following free boundary problem which was  discussed in \cite{NMKAMWG} in connection with the absorption, penetration and diffusion-induced swelling in dense and foamed rubbers.  
Let $[0,s(t)]$ be a region occupied by a solvent (e.g. water, tea) occupying the one-dimensional pore $[0,\infty)$, where $t$ is the time variable,   
$s = s(t)$ is the position of the moving interface, while $u = u(t, z)$ is the content of the diffusant situated at the position $z \in [0,s(t)]$.
The function $u(t,z)$ acts in the non-cylindrical region $Q_s(T)$ given by
\begin{align*}
& Q_s(T):=\{(t, z) | 0<t<T, \ 0<z<s(t) \}.
\end{align*}

Our free boundary problem $(\mbox{P})(u_0, s_0, b)$ reads: Find the pair $(u, s)$ satisfying
\begin{align}
& u_t-u_{zz}=0 \mbox{ for }(t, z)\in Q_s(T), \label{1-1}\\
& -u_z(t, 0)=\beta(b(t)-\gamma u(t, 0)) \mbox{ for }t\in(0, T), \label{1-2}\\
& -u_z(t, s(t))=u(t, s(t))s_t(t) \mbox{ for }t\in (0, T), \label{1-3}\\
& s_t(t)=a_0 \sigma(u(t, s(t))) \mbox{ for }t\in (0, T), \label{1-4}\\
& s(0)=s_0, u(0, z)=u_0(z) \mbox{ for }z \in [0, s_0], \label{1-5}
\end{align}
where $\beta$, $\gamma$ and $a_0$ are given positive constants, $b$ is a given threshold function defined on $[0, T]$, while $s_0$ and $u_0$ are the corresponding initial data. In (\ref{1-4}), $\sigma$ is a function on $\mathbb{R}$ given by 
$$
\sigma(r)=
\begin{cases}
r \quad \mbox{ if }r\geq 0,\\
0 \quad \mbox{ if }r<0.
\end{cases}
$$
In \cite{Visintin}, A. Visintin refers to this type of problems as free or moving boundary problems with {\em kinetic} boundary condition. The reason for calling this way is linked to the fact that relation (\ref{1-4}) is an explicit description of the speed of the free boundary. Note also that, in Refs. \cite{KA, KA2},  the authors have considered the mathematical analysis of a similar problem to $(\mbox{P})(u_0, s_0, b)$ related to  water-induced swelling in porous materials, viz.
\begin{align}
& u_t-u_{zz}=0 \mbox{ for }(t, z)\in (0, T)\times (a, s(t)), \label{1-6}\\
& -u_z(t, a)=\hat{\beta}(h(t)-\gamma u(t, a)) \mbox{ for }t\in(0, T), \label{1-7}\\
& -u_z(t, s(t))=u(t, s(t))s_t(t) \mbox{ for }t\in (0, T), \label{1-8}\\
& s_t(t)=a_0 (u(t, s(t))-\varphi(s(t))) \mbox{ for }t\in (0, T), \label{1-9}\\
& s(0)=\hat{s}_0, u(0, z)=\hat{u}_0(z) \mbox{ for }z \in [a, \hat{s}_0].\label{1-10}
\end{align}
In this context, $a$ is a positive constant, $h$ is a given non-negative function on $[0, T]$, while $\hat{s}_0$ and $\hat{u}_0$ are the initial data such that $\hat{s}_0>a$. Also, $\hat{\beta}$ and $\varphi$ are continuous  functions  on  $\mathbb{R}$ such that $\hat{\beta}(r)>0$  and $\varphi(r)>0$ for $r>0$, and $\hat{\beta}(r)=\varphi(r)=0$ for $r\leq 0$. 
 Denote the above problem $\{(\ref{1-6})-(\ref{1-10})\}$ by $(\hat{\mbox{P}})(\hat{u}_0, \hat{s}_0, h)$. In \cite{KA} in was assumed that $\varphi$ is conveniently small and $\hat{u}_0\in H^1(a, \hat{s}_0)$ such that $\varphi(a)\leq \hat{u}_0 \leq h^*/\gamma$ on $[a, \hat{s}_0]$. Such conditions ensure the existence of a locally-in-time solution $(u, s)$ to $(\hat{\mbox{P}})(\hat{u}_0, \hat{s}_0, h)$ on $[0, T_0]$ 
such that $\varphi(a)\leq u \leq h^*/\gamma$ on $Q^a_s(T_0)$ for some $0<T_0\leq T$, where $h^*$ is a upper bound of $h$.
In \cite{KA2}, relying on  the same assumptions as in \cite{KA}, the authors have constructed a globally-in-time solution $(u, s)$ to $(\hat{\mbox{P}})(\hat{u}_0, \hat{s}_0, h)$ on $[0, T]$ such that $\varphi(a)\leq u \leq h^*/\gamma$ on $Q^a_s(T)$. 
What concerns the  large time behavior of solutions to $(\hat{\mbox{P}})(\hat{u}_0, \hat{s}_0, h)$, one reports in \cite{KA2} that the following situation holds: 
\begin{align*}
\mbox{ if } \lim_{t\to \infty} \int_0^t \hat{\beta}(h(\tau)-\gamma u(\tau, a))d\tau =\infty, \mbox{ then} \lim_{t\to \infty} s(t)=\infty.
\end{align*}

One of our concrete aims here is to construct a global-in-time solution of $(\mbox{P})(u_0, s_0, b)$. As anticipated,  the key to the proof is to establish the strictly positivity for the free boundary. To this end, we consider the free boundary condition (\ref{1-4}),  which should be seen  as  $\varphi \equiv 0$ in (\ref{1-9}) in the models proposed in \cite{KA, KA2}. This is a simplification of the modeling setting which is convenient for mathematical analysis purposes. Furthermore, we adopt  the positive part in (\ref{1-4}), and  hence, we can easily show that the free boundary $s(t)$ is  indeed strictly positive and  the expected global  existence is now reachable. We will investigate elsewhere to which extent such structural restrictions can be relaxed.

Moreover, we establish that our free boundary grows up, namely, it is unbounded. In order to obtain a control on the growth of the free boundary, the mass conservation law (respectively, the momentum balance law)  are effective ingredients in case the boundary condition at the fixed boundary of Neumann (respectively, Dirichlet) type, see for instance  \cite{Cannon}. In the present setting, we impose a Robin boundary condition at the fixed boundary and the usual approach does not work well. Hence, the
rationale beyond showing that $s(t) \to \infty$ as $t \to \infty$ is as follows: If the free boundary is bounded, then we can obtain some uniform-in-time estimates for the target solution, and  consequently, this solution $u(t)$ converges towards the stationary solution of our problem. It is worthwhile to note here that the stationary problem still contains a free boundary condition. However, as a consequence of our uniform estimates, the solution never satisfies the stationary free boundary condition. Thus we can prove the large time behavior by detecting a  contradiction. The idea of applying a contradiction argument concerned with a stationary solution was already applied in \cite{AIIY}. The quest for growth (convergence) rates of  this kind of kinetically-driven free boundaries was completed in the series of papers \cite{AM1,AM2,AM3}. For the problem at hand, proving  quantitative estimates on the growth rate of the free boundary is currently an open problem. 

%
%
%
%

\section{Notation, assumptions and results}
\label{na}

In this paper, we use the following notations. We denote by $| \cdot |_X$ the norm for a Banach space $X$. The norm and the inner product of a Hilbert space $H$ are
 denoted by $|\cdot |_H$ and $( \cdot, \cdot )_H$, respectively. Particularly, for $\Omega \subset {\mathbb R}$, we use the notation of the usual Hilbert spaces $L^2(\Omega)$, $H^1(\Omega)$ and $H^2(\Omega)$.
Throughout this paper, we assume the following parameters and functions:

(A1) $a_0$, $\gamma$, $\beta$ and $T$ are positive constants.

(A2) $b\in W^{1,2}(0, T)$ with $b_*\leq b \leq b^*$ on $(0, T)$, where $b_*$ and $b^*$ are positive constants.

(A3) $s_0>0$ and $u_0\in H^1(0, s_0)$ such that $0 \leq u_0 \leq b^*/ \gamma $ on $[0, s_0]$.

Next, we define our concept of solution to (P)$(u_0, s_0, b)$ on $[0,T]$ in the following way:
\begin{defin}
\label{def0}
For $T>0$, let $s$ be a function on $[0, T]$ and $u$ be a function on $Q_s(T)$. We call  
the pair $(s, u)$  a solution to (P)$(u_0, s_0, b)$ on $[0, T]$ if the following conditions (S1)-(S6) hold:

(S1) $s\in W^{1,\infty}(0, T)$, $0<s$ on $[0, T]$, $u\in L^{\infty}(Q_s(T))$, $u_t$, $u_{zz}\in L^2(Q_s(T))$ and $t \in [0, T] \to |u_z(t,  \cdot)|_{L^2(0, s(t))}$ is bounded;

(S2) $u_t-u_{zz}=0$ on $Q_s(T)$;

(S3) $-u_z(t, 0)=\beta(b(t)-\gamma u(t, 0))$ for a.e. $t\in [0, T]$;

(S4) $-u_z(t, s(t))=u(t, s(t))s_t(t)$ for a.e. $t\in [0, T]$;

(S5) $s_t(t)=a_0 \sigma(u(t, s(t))) $ for a.e. $t\in [0, T]$;

(S6) $s(0)=s_0$ and $u(0, z)=u_0(z)$ for $z \in [0, s_0]$.
\end{defin}

The first result of this paper is concerned with the existence and uniqueness of a  locally-in-time solution  in the sense of Definition \ref{def0} to the problem (P)$(u_0, s_0, b)$.
\begin{thm}
\label{t1}
Let $T>0$. If (A1)-(A3) hold, then there exists 
$T^*  \in  (0, T]$ 
such that (P)$(u_0, s_0, b)$ has a unique solution $(s, u)$ on $[0, T^*]$ satisfying $0 \leq u \leq b^* /\gamma$ on $Q_s(T^*)$.
\end{thm}
To prove Theorem \ref{t1}, we transform (P)$(u_0, s_0, b)$, initially posed in a non-cylindrical domain, to  a cylindrical domain. Let $T>0$. For given $s\in W^{1,2}(0, T)$ with $s(t)>0$ on $[0, T]$, we introduce the following new function obtained by the change of variables and fix the moving domain:
\begin{align}
\label{2-0}
\tilde{u}(t, y)=u(t, ys(t)) \mbox{ for } (t, y)\in Q(T):=(0, T)\times (0, 1)
\end{align}
By using the function $\tilde{u}$,
(P)$(u_0, s_0, b)$ becomes the following problem (PC)$(\tilde{u}_0, s_0, b)$ on the cylindrical domain $Q(T)$:
\begin{align}
& \tilde{u}_t(t, y)-\frac{1}{s^2(t)}\tilde{u}_{yy}(t, y)=\frac{ys_t(t)}{s(t)}\tilde{u}_y(t, y) \mbox{ for }(t, y)\in Q(T), \label{2-1}\\
& -\frac{1}{s(t)}\tilde{u}_y(t, 0)=\beta(b(t)-\gamma \tilde{u}(t, 0)) \mbox{ for }t\in(0, T), \label{2-2}\\
& -\frac{1}{s(t)}\tilde{u}_y(t, 1)=\tilde{u}(t, 1)s_t(t) \mbox{ for }t\in (0, T),  \label{2-3}\\
& s_t(t)=a_0 \sigma(\tilde{u}(t, 1)) \mbox{ for }t\in (0, T), \label{2-4}\\
& s(0)=s_0, \label{2-5}\\
& \tilde{u}(0, y)=u_0(y s(0))(:=\tilde{u}_0(y)) \mbox{ for }y \in [0, 1]. \label{2-6}
\end{align}

\begin{defin}
\label{def1}
For $T>0$, let $s$ be a function on $[0, T]$ and $\tilde{u}$ be a function on $Q(T)$, respectively. We call that a pair $(s, \tilde{u})$ is a solution of $(\mbox{P})(\tilde{u}_0, s_0, b)$ on $[0, T]$ if the conditions (S'1)-(S'2) hold:

(S'1) $s \in W^{1, \infty}(0, T)$, $s>0$ on $[0, T]$, $\tilde{u}\in W^{1,2}(Q(T))\cap L^{\infty}(0, T; H^1(0, 1))\cap L^2(0, T;H^2(0, 1))$.

(S'2) (\ref{2-1})--(\ref{2-6}) hold.
\end{defin}

Here, we introduce the following function space: For $T>0$, we put $V(T)=L^{\infty}(0, T; L^2(0, 1))\cap L^2(0, T; H^1(0, 1))$ and $|z|_{V(T)}=|z|_{L^{\infty}(0, T; L^2(0, 1))} + |z_y|_{L^2(0, T; L^2(0, 1))}$ for $z\in V(T)$. Note that $V(T)$ is a Banach space with the norm $|\cdot|_{V(T)}$. 

Now, we state the existence and uniqueness of a locally-in-time solution of $(\mbox{PC})(\tilde{u}_0, s_0, b)$. 

\begin{thm}
\label{t2}
Let $T>0$. If (A1)-(A3) hold, then there exists 
$T^* \in (0, T]$ 
such that $(\mbox{PC})(\tilde{u}_0, s_0, b)$ has a unique solution $(s, \tilde{u})$ on $[0, T^*]$. 
\end{thm}

By Theorem \ref{t2}, we see that for a solution $(s, \tilde{u})$ of  $(\mbox{PC})(\tilde{u}_0, s_0, b)$ on $[0, T^*]$,  a pair of the function $(s, u)$ with the variable
\begin{align}
u(t, z):=\tilde{u}\left(t, \frac{z}{s(t)}\right) \mbox{ for } z\in [0, s(t)]
\label{2-7}
\end{align}
is a solution of $(\mbox{P})(u_0, s_0, b)$ on $[0, T^*]$. Moreover, by proving that $(s, u)$ satisfies $0\leq u\leq b^*/\gamma$ on $Q_s(T^*)$, the pair $(s, u)$ is a desired solution of $(\mbox{P})(u_0, s_0, b)$ on $[0, T^*]$ which leads to Theorem \ref{t1}. 

The second result of this paper is the existence and uniqueness of a globally-in-time solution in the sense of Definition \ref{def0} to the problem (P)$(u_0, s_0, b)$.

\begin{thm}
\label{t3}
Let $T>0$. If (A1)-(A3) hold, then $(\mbox{P})(u_0, s_0, b)$ has a unique solution $(s, u)$ on $[0, T]$ satisfying $0 \leq u \leq b^*/\gamma$ on $Q_s(T)$.
\end{thm}

Throughout 
Sections \ref{AP} and \ref{LE}, 
we show Theorem \ref{t1} by proving Theorem \ref{t2} and the boundedness of a solution of $(\mbox{P})(u_0, s_0, b)$. In Section \ref{GE}, we give a proof of Theorem \ref{t3}. 
In the last section, we discuss the large time behavior of a solution of $(\mbox{P})(u_0, s_0, b)$ as $t \to \infty$. In fact, we obtain the result that $s\to \infty$ as $t\to \infty$. The precise statement is stated as Theorem \ref{b-2}.

%
%
%
%
%
%
\section{Auxiliary Problem}
\label{AP}

In this section, we prove Theorem \ref{t2} on the existence and uniqueness of a locally-in-time  solution of $(\mbox{PC})(\tilde{u}_0, s, b)$. 
To do so, we introduce the following auxiliary problem $(\mbox{AP})(\tilde{u}_0, s, b)$: For $T>0$, $s_0>0$ and given $s\in W^{1, 2}(0, T)$ with $s(0)=s_0$ and $s \geq s_0 $ on $[0, T]$, 
\begin{align}
& \tilde{u}_t(t, y)-\frac{1}{s^2(t)}\tilde{u}_{yy}(t, y)=\frac{ys_t(t)}{s(t)} \tilde{u}_y(t, y) \mbox{ for }(t, y)\in Q(T), \label{3-1}\\
& -\frac{1}{s(t)}\tilde{u}_y(t, 0)=\beta(b(t)-\gamma \tilde{u}(t, 0)) \mbox{ for }t\in(0, T), \label{3-2}\\
& -\frac{1}{s(t)}\tilde{u}_y(t, 1)=a_0 \tilde{u}(t, 1)\sigma(\tilde{u}(t, 1))\mbox{ for }t\in (0, T),  \label{3-3}\\
& \tilde{u}(0, y)=\tilde{u}_0(y) \mbox{ for }y \in [0, 1],  \label{3-4}
\end{align}
where $\sigma$ is the same function as in (\ref{1-4}).

First of all, to solve $(\mbox{AP})(\tilde{u}_0, s, b)$, for given  $s\in W^{1,\infty}(0, T)$ with $s(0)=s_0$ and $s \geq s_0$ on $[0, T]$ and $f\in L^2(0, T; H^1(0, 1))$, we consider the problem $(\mbox{AP}1)(\tilde{u}_0, s, f, b)$:
\begin{align*}
& \tilde{u}_t(t, y)-\frac{1}{s^2(t)}\tilde{u}_{yy}(t, y)=\frac{ys_t(t)}{s(t)} f_y(t, y) \mbox{ for }(t, y)\in Q(T), \\
& -\frac{1}{s(t)}\tilde{u}_y(t, 0)=\beta(b(t)-\gamma \tilde{u}(t, 0)) \mbox{ for }t\in(0, T), \\
& -\frac{1}{s(t)}\tilde{u}_y(t, 1)=a_0 \tilde{u}(t, 1)\sigma(\tilde{u}(t, 1)) \mbox{ for }t\in (0, T), \\
& \tilde{u}(0, y)=\tilde{u}_0(y) \mbox{ for }y \in [0, 1].
\end{align*}

Now, we define a family $\{ \psi^t \}_{t\in[0, T]}$ of time-dependent functionals $\psi^t: L^2(0, 1) \to \mathbb{R}\cup \{+\infty \}$
for $t\in [0, T]$ as follows: 
$$
\psi^t(u):=\begin{cases}
             \displaystyle{\frac{1}{2s^2(t)} \int_0^1|u_y(y)|^2 dy} + \displaystyle{\frac{1}{s(t)}\int_0^{u(1)} a_0 \xi \sigma(\xi) d\xi}\\
             -\displaystyle{\frac{1}{s(t)}\int_0^{u(0)}\beta(b(t)-\gamma\xi)d\xi} \mbox{ if } u\in D(\psi^t),\\
             +\infty \quad \rm{otherwise},
             \end{cases}
$$
where $D(\psi^t)=\{z\in H^1(0, 1) | z\geq 0 \mbox{ on } [0, 1] \}$ for $t \in [0, T]$.
Here, we show the property of $\psi^t$.

%
%
%
%
\begin{lem}
\label{lem1}
Let $s\in W^{1,2}(0, T)$ with $s(0)=s_0$ and $s \geq s_0$ on $[0, T]$ and assume (A1)-(A3). Then the following statements hold:
\begin{itemize}
\item[(1)] There exists positive constants $C_0$ and $C_1$ such that the following inequalities hold:
\begin{align*}
(i) & \ |u(y)|^2 \leq C_0 \psi^t(u)+C_1 \mbox{ for }u\in D(\psi^t), 
y=0, 1 \mbox{ and } t \in [0,T], 
\\
(ii) & \ \frac{1}{2s^2(t)}|u_y|^2_{L^2(0, 1)} \leq C_0 \psi^t(u)+C_1 \mbox{ for }u\in D(\psi^t)
\mbox{ and } t \in [0,T]. 
\end{align*}
\item[(2)] For $t\in [0, T]$, the functional $\psi^t$ is proper, lower semi-continuous, and  convex  on $L^2(0, 1)$.
\end{itemize}
\end{lem}

\begin{proof}
First, we note that for $t\in [0, T]$ if $u\in D(\psi^t)$, then $u(0)$ is non-negative.
Let $t\in [0, T]$ and $u\in D(\psi^t)$. Then, it holds 
\begin{align}
 - \frac{1}{s(t)} \int_0^{u(0)} \beta (b(t)-\gamma \xi) d\xi 
& =\frac{\beta}{s(t)} \left [ \frac{\gamma}{2} u^2(0) -b(t) u(0)\right] \nonumber \\
\geq \frac{\beta \gamma}{2l} u^2(0)-\frac{\beta b^*}{s_0} u(0) & \geq \frac{\beta \gamma}{4l} u^2(0)-\frac{\beta l}{\gamma} \left(\frac{b^*}{s_0} \right)^2.
\label{3-5}
\end{align}
Since the second term of the right-hand side of $\psi^t$ is positive, by (\ref{3-5}), we have that 
\begin{align}
\psi^t(u) &\geq \frac{1}{2s^2(t)} \int_0^1|u_y(y)|^2 dy + \frac{\beta \gamma}{4l} u^2(0)-\frac{\beta l}{\gamma} \left(\frac{b^*}{s_0}\right)^2.
\label{3-6}
\end{align}
Also, it holds that
\begin{align}
|u(1)|^2 & =\biggl | \int_0^1 u_y(y) dy + u(0) \biggr|^2 \leq 2\left (\int_0^1 |u_y(y)|^2 dy + |u(0)|^2 \right) \nonumber \\
& \leq 2\left (\frac{2l^2}{2s^2(t) }\int_0^1 |u_y(y)|^2 dy + |u(0)|^2 \right).
\label{3-6-1}
\end{align}
Therefore, by (\ref{3-6}) and (\ref{3-6-1}) we see that the statement (1) of Lemma \ref{lem1} holds.

Next, we prove statement (2).  For $t\in [0, T]$ and $r\in \mathbb{R}$, put
\begin{align*}
& g_1(s(t), r)=\frac{1}{s(t)}\int_0^r a_0\xi \sigma(\xi) d\xi, \\
& g_2(s(t), b(t), r)=-\frac{1}{s(t)}\int_0^r \beta(b(t)-\gamma \xi)d\xi.
\end{align*}
Then, by $0<s(t)$ we see that 
\begin{align*}
& \frac{\partial^2}{\partial r^2} g_1(s(t), r) =\frac{2a_0}{s(t)} r >0 \mbox{ for }r>0, \\
& \frac{\partial^2}{\partial r^2}g_2(s(t), b(t), r)= \frac{\beta \gamma}{s(t)}>0 \mbox{ for } r\in \mathbb{R}.
\end{align*}
This means that $\psi ^t$ is convex on $L^2(0, 1)$.
Also, by using Lemma \ref{lem1} and Sobolev's embedding $H^1(0, 1) \hookrightarrow C([0, 1])$ in one dimensional case, it is easy to prove that the level set of $\psi^t$ is closed in $L^2(0, 1)$ which leads to the lower semi-continuity of $\psi^t$. Thus, we see that statement (2) holds.
\end{proof}

By Lemma \ref{lem1} we obtain the existence of a solution to $(\mbox{AP}1)(\tilde{u}_0, s, f, b)$.
%
%
%
%
\begin{lem}
\label{lem2}
Let $T>0$ and $s_0>0$. If (A1)-(A3) hold, then, for given $s\in W^{1,2}(0, T)$ with $s(0)=s_0$ and $s \geq s_0 $ on $[0, T]$ and  $f\in L^2(0, T; H^1(0, 1))$, 
the problem $(\mbox{AP}1)(\tilde{u}_0, s, f, b)$ admits a unique solution $\tilde{u}$ on $[0, T]$ such that $\tilde{u}\in W^{1,2}(Q(T))\cap L^{\infty}(0, T;H^1(0, 1))$ with $\tilde{u}\geq 0$ on $Q(T)$. Moreover, the function $t\to \psi^t(\tilde{u}(t))$ is absolutely continuous on $[0, T]$.
\end{lem}

\begin{proof}
By Lemma \ref{lem1}, for $t\in [0, T]$ $\psi^t$ is a proper lower semi-continuous convex function on $L^2(0, 1)$.
From the definition of the subdifferential of $\psi^t$, for $t\in [0, T]$, $z^*\in \partial \psi^t(u)$ is characterized by
$u$, $z^*\in L^2(0, 1)$, 
\begin{align*}
z^*=-\frac{1}{s^2(t)} u_{yy} \mbox{ on } (0, 1), \quad -\frac{1}{s(t)}u_y(0)=\beta(b(t)-\gamma u(0)), \quad -\frac{1}{s(t)}u_y(1)=a_0 u(1)\sigma(u(1)).
\end{align*}
Namely, $\partial \psi^t$ is single-valued. Also, we see that there exists a positive constant $C_3$ such that for each $t_1$, $t_2\in [0, T]$ with $t_1\leq t_2$, and for any $u\in D(\psi^{t_1})$, there exists $\bar{u}\in D(\psi^{t_2})$ such that
\begin{align}
&|\psi^{t_2}(\bar{u})-\psi^{t_1}(u)|\leq C_3(|s(t_1)-s(t_2)|+|b(t_1)-b(t_2)|)(1+|\psi^{t_1}(u)|). \label{3-8}
\end{align}
Indeed, by taking $\bar{u}:=u$ and using (i) and (ii) of Lemma \ref{lem1}, we can find $C_3>0$ such that (\ref{3-8}) holds. 
Now, we consider the following Cauchy problem (CP):
$$ 
\begin{cases}
\tilde{u}_t+\partial \psi^t(\tilde{u}(t)) =\frac{y s_t(t)}{s(t)}f_y(t) \mbox{ in }L^2(0, 1), \\
\tilde{u}(0, y)=\tilde{u}_0(y) \mbox{ for }y\in [0, 1].
\end{cases}
$$
Since $\frac{ys_t }{s}f_y \in L^2(0, T; L^2(0, 1))$, the general theory of evolution equations governed by time dependent subdifferentials (cf. \cite{kenmochi}) guarantees that 
(CP) has a non-negative solution $\tilde{u}$ on $[0, T]$ such that $\tilde{u}\in W^{1,2}(Q(T))$, $\psi^t(\tilde{u}(t)) \in L^{\infty}(0, T)$ and $t \to \psi^t(\tilde{u}(t))$ is absolutely continuous on $[0, T]$. This implies that $\tilde{u}$ is a unique solution of $(\mbox{AP}1)(\tilde{u}_0, s, f, b)$ on $[0, T]$.
\end{proof}

%
%
%
%
\begin{lem}
\label{lem3}
Let $T>0$,  $s_0>0$ and $s\in W^{1,\infty}(0, T)$ with $s(0)=s_0$ and $s \geq s_0$ on $[0, T]$. If (A1)-(A3) hold, then, $(\mbox{AP})(\tilde{u}_0, s, b)$ has a unique solution $\tilde{u}$ on $[0, T]$ such that $\tilde{u}\in W^{1,2}(Q(T))\cap L^{\infty}(0, T;H^1(0, 1))$.
\end{lem}

\begin{proof}
First, we define a solution operator $\Gamma_T(f)=\tilde{u}$, where
$\tilde{u}$ is the unique solution of $(\mbox{AP}1)(\tilde{u}_0, s, f, b)$ for given $f\in V(T)$.
Now, for $i=1,2$ we put $\Gamma_T(f_i)=\tilde{u}_i$ and $f=f_1-f_2$ and $\tilde{u}=\tilde{u}_1-\tilde{u}_2$.
Then, we have that 
\begin{align}
& \frac{1}{2}\frac{d}{dt}|\tilde{u}(t)|^2_{L^2(0, 1)} - \int_0^1 \frac{1}{s^2(t)}\tilde{u}_{yy}(t) \tilde{u}(t) dy = \int_0^1 \frac{y s_t(t)}{s(t)}f_{y}(t)\tilde{u}(t) dy.\label{3-9}
\end{align}
Using the boundary condition, it holds that 
\begin{align}
& - \int_0^1 \frac{1}{s^2(t)}\tilde{u}_{yy}(t) \tilde{u}(t) dy \nonumber \\
& = \frac{1}{s^2(t)}\left(-\tilde{u}_y(t, 1)\tilde{u}(t, 1)+\tilde{u}_y(t, 0)\tilde{u}(t, 0) + \int_0^1|\tilde{u}_y(t)|^2 dy\right) \nonumber \\
& =\frac{a_0}{s(t)}\biggl(\tilde{u}_1(t, 1)\sigma(\tilde{u}_1(t, 1))-\tilde{u}_2(t, 1)\sigma(\tilde{u}_2(t, 1))\biggr)\tilde{u}(t, 1) \nonumber \\
& -\frac{1}{s(t)}\biggl(\beta(b(t)-\gamma \tilde{u}_1(t, 0))-\beta(b(t)-\gamma \tilde{u}_2(t, 0))\biggr)\tilde{u}(t, 0) + \frac{1}{s^2(t)}\int_0^1|\tilde{u}_y(t)|^2 dy.
\label{3-10}
\end{align}
Since the function $r\sigma(r)$ is monotone for $r\in \mathbb{R}$, the first term of the right-hand side of $(\ref{3-10})$ is non-negative. 
The second term of the right-hand side of (\ref{3-10}) is also non-negative, hence we see that 
\begin{align}
- \int_0^1 \frac{1}{s^2(t)}\tilde{u}_{yy}(t) \tilde{u}(t) dy \geq \frac{1}{s^2(t)}\int_0^1|\tilde{u}_y(t)|^2 dy.
\label{3-11}
\end{align}
Accordingly, by (\ref{3-9})-(\ref{3-11}), we have that 
\begin{align}
\frac{1}{2}\frac{d}{dt}|\tilde{u}(t)|^2_{L^2(0, 1)}  + \frac{1}{s^2(t)}\int_0^1|\tilde{u}_y(t)|^2 dy \leq  \int_0^1 \frac{ys_t(t)}{s(t)}f_{y}(t) \tilde{u}(t)dy.
\label{3-12}
\end{align}
Here, using H{\" o}lder's inequality, it holds that 
\begin{align}
\int_0^1 \frac{y s_t(t)}{s(t)}f_{y}(t) \tilde{u}(t)dy \leq \frac{|s_t|_{L^{\infty}(0, T)}}{s_0}|\tilde{u}(t)|_{L^2(0, 1)} |f_y(t)|_{L^2(0, 1)}. 
\label{3-13}
\end{align}
Let $T_1\in (0, T]$. Then, by putting $l=\mbox{max}_{0\leq t\leq T}|s(t)|$ and integrating (\ref{3-12}) with (\ref{3-13}) over $[0, t]$ for any $t\in [0, T_1]$ we obtain that 
\begin{align}
& \frac{1}{2}|\tilde{u}(t)|^2_{L^2(0, 1)}  + \frac{1}{2l^2} \int_0^{t} \int_0^1|\tilde{u}_y(\tau)|^2 dy d\tau \nonumber \\
\leq & \frac{|s_t|_{L^{\infty}(0, T)}}{s_0} |\tilde{u}|_{L^{\infty}(0, T_1; L^2(0, 1))} T^{1/2}_1 \left( \int_0^{T_1} |f_y(\tau)|^2_{L^2(0, 1)} d\tau \right)^{1/2}\nonumber \\
\leq & \frac{|s_t|_{L^{\infty}(0, T)}}{s_0} |\tilde{u}|_{V(T_1)} T^{1/2}_1 |f|_{V(T_1)} 
\label{3-14}
\end{align}
Therefore, by putting $\delta=\mbox{min}\{1/2, 1/2l^2\}$ in (\ref{3-14}) we have that 
\begin{align*}
& \delta |\tilde{u}|_{V(T_1)} \leq \frac{|s_t|_{L^{\infty}(0, T)}}{s_0} T^{1/2}_1 |f|_{V(T_1)} \mbox{ for }T_1\in (0, T].
\end{align*}
From this result, we infer that for some $T_1\leq T$ such that $\Gamma_{T_1}$ is a contraction mapping in $V(T_1)$. Therefore, by Banach's fixed point theorem there exists $\tilde{u}\in V(T_1)$ such that
$\Gamma_{T_1}(\tilde{u})=\tilde{u}$ which implies $\tilde{u}$ is a solution of $(\mbox{AP})(\tilde{u}_0, s, b)$ on $[0, T_1]$. Since $T_1$ is independent of the choice of initial data of $\tilde{u}$, by repeating the argument of the local existence result, we can extend the solution $\tilde{u}$ beyond $T_1$. Thus, we prove that Lemma \ref{lem3} holds. 
\end{proof}

Next,  for given $s\in W^{1,2}(0, T)$ with $s(0)=s_0$ and $s \geq s_0$ on $[0, T]$, we construct a solution to $(\mbox{AP})(\tilde{u}_0, s, b)$.

%
%
%
%
\begin{lem}
\label{lem4}
Let $T>0$ and  $s_0>0$. If (A1)-(A3) hold, then, for given $s\in W^{1,2}(0, T)$ with $s(0)=s_0$ and $s \geq s_0$ on $[0, T]$, $(\mbox{AP})(\tilde{u}_0, s, b)$ has a unique solution $\tilde{u}$ on $[0, T]$.
\end{lem}

\begin{proof}
For given $s\in W^{1, 2}(0, T)$ with $s(0)=s_0$ and $s\geq s_0$ on $[0, T]$, we choose a sequence $\{s_n\} \subset W^{1, \infty}(0, T)$ and $l>0$ satisfying $s_0\leq s_n\leq l$ on $[0, T]$ for each $n\in \mathbb{N}$,
$s_n \to s$ in $W^{1,2}(0, T)$ as $n \to \infty$. By Lemma \ref{lem3} we can take a sequence $\{\tilde{u}_n\}$ of solutions to $(\mbox{AP})(\tilde{u}_0, s_n, b)$ on $[0, T]$. Then, we see that $t \to \psi^t(\tilde{u}_n(t))$ is absolutely continuous on $[0, T]$ so that $t \to \frac{1}{s^2_n(t)}|\tilde{u}_{ny}(t)|^2_{L^2(0, 1)}$ is continuous on $[0, T]$.
First, it holds that 
\begin{align*}
& \frac{1}{2}\frac{d}{dt}|\tilde{u}_n(t)|^2_{L^2(0, 1)} -\int_0^1 \frac{1}{s^2_n(t)}\tilde{u}_{nyy}(t) \tilde{u}_n(t) dy = \int_0^1 \frac{y s_{nt}(t)}{s_n(t)} \tilde{u}_{ny}(t) \tilde{u}_n(t) dy.
\end{align*}
For the second term in the left-hand side, we have that
\begin{align*}
& -\int_0^1 \frac{1}{s^2_n(t)}\tilde{u}_{nyy}(t) \tilde{u}_n(t) dy \\
=& \frac{a_0}{s_n(t)}\tilde{u}_n(t, 1)\sigma(\tilde{u}_n(t, 1))\tilde{u}_n(t, 1)-\frac{1}{s_n(t)}\beta(b(t)-\gamma \tilde{u}_n(t, 0))\tilde{u}_n(t, 0) + \frac{1}{s^2_n(t)}\int_0^1|\tilde{u}_{ny}(t)|^2 dy \\
\geq & -\frac{1}{s_n(t)}\beta(b(t)-\gamma \tilde{u}_n(t, 0))\tilde{u}_n(t, 0) + \frac{1}{s^2_n(t)}\int_0^1|\tilde{u}_{ny}(t)|^2 dy.
\end{align*}
Hence, we obtain that
\begin{align}
& \frac{1}{2}\frac{d}{dt}|\tilde{u}_n(t)|^2_{L^2(0, 1)} + \frac{1}{s^2_n(t)}\int_0^1|\tilde{u}_{ny}(t)|^2 dy \nonumber \\
\leq &  \int_0^1 \frac{y s_{nt}(t)}{s_n(t)} \tilde{u}_{ny}(t) \tilde{u}_n(t) dy + \frac{1}{s_n(t)}\beta(b(t)-\gamma \tilde{u}_n(t, 0))\tilde{u}_n(t, 0) \mbox{ for } t\in [0, T].
\label{3-15}
\end{align}
Using Young's inequality, we have that
\begin{align}
\label{3-15-1}
\int_0^1 \frac{y s_{nt}(t)}{s_n(t)} \tilde{u}_{ny}(t) \tilde{u}_n(t) dy \leq \frac{1}{2s^2_n(t)}\int_0^1|\tilde{u}_{ny}(t)|^2dy + \frac{|s_{nt}(t)|^2}{{2}}\int_0^1|\tilde{u}_n(t)|^2dy.
\end{align}
Here, by Sobolev's embedding theorem in one dimension, we note that it holds that 
\begin{align}
\label{1d}
|v(t, y)|^2 \leq C_e |v(t)|_{H^1(0, 1)} |v(t)|_{L^2(0, 1)} \mbox{ for } v\in H^1(0, 1) \mbox{ and } y\in [0, 1], 
\end{align}
where $C_e$ is a positive constant defined from Sobolev's embedding theorem. 
By (\ref{1d}) and $s_0\leq s_n$ on $[0, T]$ we get  
\begin{align}
\label{3-15-2}
& \frac{1}{s_n(t) }\beta(b(t)-\gamma \tilde{u}_n(t, 0))\tilde{u}_n(t, 0) \leq \frac{\beta b^*}{s_n(t)}|\tilde{u}_n(t, 0)| \nonumber \\
\leq & \frac{\beta b^*C_e}{2s_n(t)}\biggl(|\tilde{u}_{ny}(t)|_{L^2(0, 1)}|\tilde{u}_n(t)|_{L^2(0, 1)} + |\tilde{u}_n(t)|^2_{L^2(0, 1)}\biggr) +
\frac{\beta b^*}{2 s_n(t)} \nonumber \\
\leq & \frac{1}{4s^2_n(t)}|\tilde{u}_{ny}(t)|^2_{L^2(0, 1)} + \left( (\beta b^* C_e)^2+ \frac{\beta b^* C_e}{2s_0}\right)|\tilde{u}_n(t)|^2_{L^2(0, 1)} + \frac{\beta b^*}{2s_0}.
\end{align}
As a result, we see from (\ref{3-15})-(\ref{3-15-2}) that
\begin{align*}
& \frac{1}{2}\frac{d}{dt}|\tilde{u}_n(t)|^2_{L^2(0, 1)} + \frac{1}{4s^2_n(t)}\int_0^1|\tilde{u}_{ny}(\tau)|^2 dy \\
\leq & \left( \frac{|s_{nt}(t)|^2}{{2}} + (\beta b^* C_e)^2+ \frac{\beta b^* C_e}{2s_0} \right)|\tilde{u}_n(t)|^2_{L^2(0, 1)} + \frac{\beta b^*}{2s_0} \mbox{ for }t\in[0, T].
\end{align*}
Now, we denote $F_n(t)$ the coefficient of $|\tilde{u}_n|^2_{L^2(0, 1)}$ in the right-hand side. Then, by $s_n \leq l$ on $[0, T]$, the boundedness of $\{F_n\}$ in $L^1(0, T)$ and Gronwall's inequality
we obtain that 
\begin{align}
& \frac{1}{2}|\tilde{u}_n(t)|^2_{L^2(0, 1)} + \frac{1}{4l^2}\int_0^t |\tilde{u}_{ny}(t)|^2_{L^2(0, 1)} d\tau \nonumber \\
& \leq \left(\frac{1}{2}|\tilde{u}_0|^2_{L^2(0, 1)} + \left( \frac{\beta b^*}{2s_0}\right)T \right)e^{C} \mbox{ for }t\in [0, T].
\label{3-16}
\end{align}

Next, we put $\tilde{u}_n(t)=u_0$ for $t<0$. For each $n\in \mathbb{N}$ and $h>0$, it holds 
\begin{align}
& \int_0^1 \tilde{u}_{nt}(t) \frac{\tilde{u}_n(t)-\tilde{u}_n(t-h)}{h}dy - \int_0^1 \frac{1}{s^2_n(t)} \tilde{u}_{nyy}(t)\frac{\tilde{u}_n(t)-\tilde{u}_n(t-h)}{h}dy \nonumber \\
&=\int_0^1 \frac{y s_{nt}(t)}{s_n(t)}\tilde{u}_{ny}(t) \frac{\tilde{u}_n(t)-\tilde{u}_n(t-h)}{h}dy.
\label{3-17}
\end{align}
The second term of (\ref{3-17}) can be dealt as follows:
\begin{align*}
& - \int_0^1 \frac{1}{s^2(t)} \tilde{u}_{nyy}(t)\frac{\tilde{u}_n(t)-\tilde{u}_n(t-h)}{h}dy\\
=& -\frac{\tilde{u}_{ny}(t, 1)}{s^2_n(t)}\frac{\tilde{u}_n(t, 1)-\tilde{u}_n(t-h, 1)}{h} + \frac{\tilde{u}_{ny}(t, 0)}{s^2_n(t)}\frac{\tilde{u}_n(t, 0)-\tilde{u}_n(t-h, 0)}{h}\\
& +\int_0^1 \frac{\tilde{u}_{ny}(t)}{s^2_n(t)}\frac{\tilde{u}_{ny}(t)-\tilde{u}_{ny}(t-h)}{h}dy.
\end{align*}
We denote $I_1$, $I_2$ and $I_3$ the three terms in the last identity and estimate three terms separately. For $I_1$, using the same notation $g_1$ and $g_2$ in the proof of  Lemma \ref{lem1}, it follows that
\begin{align*}
& I_1 \geq \frac{1}{h} \frac{1}{s_n(t)} \left(\int_0^{\tilde{u}_n(t, 1)}a_0\xi \sigma(\xi) d\xi - \int_0^{\tilde{u}_n(t-h, 1)}a_0\xi \sigma(\xi) d\xi \right)\\
=& \frac{g_1(s_n(t), \tilde{u}_n(t, 1))-g_1(s_n(t-h), \tilde{u}_n(t-h, 1))}{h} + \frac{1}{h}\left(\frac{1}{s_n(t-h)}-\frac{1}{s_n(t)}\right) \int_0^{\tilde{u}_n(t-h, 1)}a_0\xi\sigma(\xi)d\xi. \\
\end{align*}
Next, for $I_2$ and $I_3$ we have that 
\begin{align*}
I_2 & \geq \frac{1}{h}\frac{1}{s_n(t)}\left(-\int_0^{\tilde{u}_n(t, 0)}\beta(b(t)-\gamma \xi)d\xi + \int_0^{\tilde{u}_n(t-h, 0)}\beta(b(t)-\gamma \xi)d\xi \right)\\
=& \frac{g_2(s_n(t), b(t), \tilde{u}_n(t, 0))-g_2(s_n(t-h), b(t-h), \tilde{u}_n(t-h, 0))}{h}\\
& + \frac{1}{h}\left(-\frac{1}{s_n(t-h)}+\frac{1}{s_n(t)}\right)\int_0^{\tilde{u}_n(t-h, 0)}\beta(b(t-h)-\gamma \xi)d\xi\\
& - \frac{1}{h}\frac{1}{s_n(t)}\int_0^{\tilde{u}_n(t-h, 0)}\biggl(\beta(b(t-h)-\gamma \xi)-\beta(b(t)-\gamma \xi)\biggr)d\xi,
\end{align*}
and 
\begin{align*}
I_3 & \geq \frac{1}{h}\frac{1}{2 s^2_n(t)}\left(\int_0^1|\tilde{u}_{ny}(t)|^2dy-\int_0^1|\tilde{u}_{ny}(t-h)|^2dy\right)\\
& =\frac{1}{h}\left(\frac{1}{2s^2_n(t)}\int_0^1|\tilde{u}_{ny}(t)|^2dy-\frac{1}{2 s^2_n(t-h)}\int_0^1|\tilde{u}_{ny}(t-h)|^2dy\right)\\
& +\frac{1}{h}\left(\frac{1}{2 s^2_n(t-h)}-\frac{1}{2 s^2_n(t)}\right) \int_0^1|\tilde{u}_{ny}(t-h)|^2dy.
\end{align*}
Combining the above three estimates and using the fact that $t \to 1/s^2_n(t)|\tilde{u}_{ny}(t)|^2$ is continuous on $[0, T]$, we obtain
\begin{align*}
& \liminf_{h\to 0}(I_1+I_2+I_3)\\
\geq & \frac{d}{dt}\psi^t(\tilde{u}_n(t)) + \frac{s_{nt}(t)}{s^2_n(t)}\int_0^{\tilde{u}_n(t, 1)}a_0\xi \sigma(\xi) d\xi + \frac{s_{nt}(t)}{s^2_n(t) } \int_0^{\tilde{u}_n(t, 0)} \beta(b(t)-\gamma \xi)d\xi \\
& + \frac{1}{s_n(t)}\int_0^{\tilde{u}_n(t, 0)}\beta b_t(t) d\xi+ \frac{s_{nt}(t)}{s^3_n(t)}\int_0^1|\tilde{u}_{ny}(t)|^2dy.
\end{align*}
Applying this result to (\ref{3-17}) and letting $h\to 0$, we observe that 
\begin{align}
& |\tilde{u}_{nt}(t)|^2_{L^2(0, 1)} + \frac{d}{dt}\psi^t(\tilde{u}_n(t)) \nonumber \\
\leq & \int_0^1\frac{y s_{nt}(t)}{s_n(t)} \tilde{u}_{ny}(t)\tilde{u}_{nt}(t)dy + \frac{|s_{nt}(t)|}{s^2_n(t)} \int_0^{\tilde{u}_n(t, 1)}a_0 \xi \sigma(\xi) d\xi + \frac{|s_{nt}(t)|}{s^2_n(t) }\biggl| \int_0^{\tilde{u}_n(t, 0)}\beta(b(t)-\gamma \xi)d\xi \biggr|\nonumber \\
& + \frac{1}{s_n(t)}\left| \int_0^{\tilde{u}_n(t, 0)}\beta b_t(t) d\xi\right| + \frac{|s_{nt}(t)|}{s^3_n(t)}\int_0^1|\tilde{u}_{ny}(t)|^2dy.
\label{3-17-1}
\end{align}
Denote $J_i (1\leq i \leq 5)$ each terms in the right-hand side of (\ref{3-17-1}).  
Using Lemma \ref{lem1} and $s_n\geq s_0$ on $[0, T]$, we  estimate each terms $J_i$ except for $i=2$ as follows:
\begin{align*}
J_1 & \leq \frac{1}{2}|\tilde{u}_{nt}(t)|^2_{L^2(0, 1)} + \frac{1}{2}\frac{|s_{nt}(t)|^2}{s^2_n(t)}|\tilde{u}_{ny}(t)|^2_{L^2(0, 1)}\\
& \leq \frac{1}{2}|\tilde{u}_{nt}(t)|^2_{L^2(0, 1)} + |s_{nt}(t)|^2 \left(C_0\psi^t(\tilde{u}_n(t))+C_1\right), \\
J_3 & \leq \frac{|s_{nt}(t)| \beta }{s^2_0}\left (b^*|\tilde{u}_n(t, 0)| + \gamma \frac{|\tilde{u}_n(t, 0)|^2}{2}\right)\\
& \leq \frac{\beta b^*}{s^2_0} \left(\frac{|s_{nt}(t)|^2}{2}+\frac{|\tilde{u}_n(t, 0)|^2}{2}\right) + \frac{\beta \gamma |s_{nt}(t)|}{2s^2_0}|\tilde{u}_n(t, 0)|^2, \\
J_4 & \leq \frac{\beta}{s_0}|b_t(t)|\tilde{u}_n(t, 0)\leq \frac{\beta}{s_0}\left(\frac{|b_{t}(t)|^2}{2}+\frac{|\tilde{u}_n(t, 0)|^2}{2}\right), \\
J_5 & \leq \frac{|s_{nt}(t)|}{s^3_n(t)}\int_0^1|\tilde{u}_{ny}(t)|^2dy \leq \frac{2|s_{nt}(t)|}{s_0}\left(C_0\psi^t(\tilde{u}_n(t))+C_1\right).
\end{align*}
For $J_2$, by the definition of $\psi^t$, we have that 
\begin{align*}
& J_2 \\
= & \frac{|s_{nt}(t)|}{s_n(t)} \left (\psi^t (\tilde{u}_n(t)) -\frac{1}{2s^2_n(t)} \int_0^1 |\tilde{u}_{ny}(t, y)|^2 dy + \frac{1}{s_n(t)} \int_0^{\tilde{u}_n(t, 0)} \beta(b(t)-\gamma \xi) d\xi \right ) \\ 
\leq & \frac{|s_{nt}(t)|}{s_n(t)} \left (\psi^t (\tilde{u}_n(t)) + \frac{1}{s_n(t)} \beta b^* \tilde{u}_n(t, 0)\right ) \\
\leq & \frac{|s_{nt}(t)|}{s_0} \left (\psi^t (\tilde{u}_n(t)) + \frac{\beta b^*}{2s_n(t)}(1+\tilde{u}^2_n(t, 0))\right)
\end{align*}
Hence, by the estimates for each $J_i$ and (\ref{3-17-1}), we obtain that
\begin{align}
\label{3-17-2}
& \frac{1}{2}|\tilde{u}_{nt}(t)|^2_{L^2(0, 1)} + \frac{d}{dt}\psi^t(\tilde{u}_n(t)) \nonumber \\
\leq & \frac{|s_{nt}(t)|}{s_0} \psi^t(\tilde{u}_n(t)) + \left (|s_{nt}(t)|^2 + \frac{2|s_{nt}(t)|}{s_0} \right) (C_0\psi^t(\tilde{u}_n(t)+C_1) \nonumber \\ 
& + \frac{\beta b^*}{s^2_0} \frac{|s_{nt}(t)|^2}{2} + \left( \frac{\beta b^*}{2s^2_0} + \frac{\beta \gamma |s_{nt}(t)|}{2s^2_0} +\frac{\beta}{2s_0} + \frac{|s_{nt}(t)|}{s_0} \frac{\beta b^*}{2s_0} \right) \tilde{u}^2_n(t, 0) \nonumber \\
& + \frac{\beta}{2s_0} |b_t(t)|^2 + \frac{\beta b^*}{2s_0} \frac{|s_{nt}(t)|}{s_0} 
\mbox{ for a.e. }t\in [0, T].
\end{align}
Here, by using (i) of Lemma \ref{lem1} we put the coefficients of $\psi^t(\tilde{u}_n(t))$ by $l_1(t)$ and otherwise by $l_2(t)$. 
Then, by the fact that $\{s_n\}$ is bounded in $W^{1,2}(0, T)$ and (A2), $l_1$, $l_2\in L^1(0, T)$. 
Now, we see from (\ref{3-17-2}) that 
\begin{align*}
\frac{1}{2}|\tilde{u}_{nt}(t)|^2_{L^2(0, 1)} + \frac{d}{dt}\psi^t(\tilde{u}_n(t)) \leq l_1(t) \psi^t(\tilde{u}_n(t)) + l_2(t) 
\mbox{ for a.e. }t\in [0, T]. 
\end{align*}
Therefore, by using Gronwall's lemma, we have that
\begin{align*}
\frac{1}{2}\int_0^t|\tilde{u}_{nt}(\tau)|^2_{L^2(0, 1)}d\tau + \psi^t(\tilde{u}_n(t)) \leq \biggl[\psi^0({\tilde{u}_0}) + \int_0^t l_2(\tau)d\tau \biggr] e^{\int_0^t l_1(\tau)d\tau}
\mbox{ for }t\in [0, T].
\end{align*}
From this result, 
we infer that the sequence $\{ \tilde{u}_{n} \}$ is bounded in $W^{1,2}(0, T;L^2(0, 1))$ and the sequence $\{ \psi^{(\cdot)}(\tilde{u}_n(\cdot))\}$ is bounded in $L^{\infty}(0, T)$. By these boundedness results and Lemma \ref{lem1}, 
we can take a sequence $\{ n_k \}\subset \{ n\}$ such that for some $\tilde{u}\in W^{1,2}(0, T; L^2(0, 1))\cap L^{\infty}(0, T; H^1(0, 1))$, $\tilde{u}_{n_k} \to \tilde{u}$ weakly in $W^{1,2}(0, T; L^2(0, 1))$, weakly -* in $L^{\infty}(0, T; H^1(0, 1))$ and in $C(\overline{Q(T)})$ as $k \to \infty$. Finally, by letting $k\to \infty$, we see that
$\tilde{u}$ is a solution of $(\mbox{AP})(\tilde{u}_0, s, b)$ on $[0, T]$.

To complete the proof, we show the uniqueness of a solution $(\mbox{AP})(\tilde{u}_0, s, b)$. 
Let $s\in W^{1, 2}(0, T)$ with $s(0)=s_0$ and $s \geq s_0$ on $[0, T]$ and $\tilde{u}_1$ and $\tilde{u}_2$ be solutions of $(\mbox{AP})(\tilde{u}_0, s, b)$ on $[0, T]$. Put $\tilde{u}=\tilde{u}_1-\tilde{u}_2$. 
Then, by (\ref{3-1}) and the same argument of the derivation of (\ref{3-12}), we have that 
\begin{align}
\label{3-24}
\frac{1}{2}\frac{d}{dt}|\tilde{u}(t)|^2_{L^2(0, 1)}  + \frac{1}{s^2(t)} |\tilde{u}_y(t)|^2_{L^2(0, 1)} 
\leq \int_0^1 \frac{y s_{t}(t)}{s(t)} \tilde{u}_{y}(t) \tilde{u}(t) dy. 
\end{align}
For the right-hand side of (\ref{3-24}), we deal as follows:
\begin{align*}
\int_0^1 \frac{y s_{t}(t)}{s(t)} \tilde{u}_{y}(t) \tilde{u}(t) dy 
\leq \frac{1}{2s^2(t)} |\tilde{u}_y(t)|^2_{L^2(0, 1)}   + \frac{|s_t(t)|^2}{2}|\tilde{u}(t)|^2_{L^2(0, 1)}. 
\end{align*}
From the above result and (\ref{3-24}) we obtain that 
\begin{align*}
\frac{1}{2}\frac{d}{dt}|\tilde{u}(t)|^2_{L^2(0, 1)}  + \frac{1}{2l^2} |\tilde{u}_y(t)|^2_{L^2(0, 1)} 
\leq  \frac{|s_t(t)|^2}{2}|\tilde{u}(t)|^2_{L^2(0, 1)}, 
\end{align*}
where $\displaystyle{l=\mbox{max}_{0\leq t\leq T}|s(t)|}$. Therefore, by 
Gronwall's lemma we have that $|\tilde{u}(t)|_{L^2(0, 1)}=0$ for $t\in [0, T]$. 
This implies that $\tilde{u}_1=\tilde{u}_2$ on $[0, T]$. Thus, Lemma \ref{lem4} is proved. 
\end{proof}

%
%
%
\section{Proof of Theorem \ref{t1}}
\label{LE}
In this section, using the results obtained in Section \ref{AP}, we establish the existence of a locally-in-time solution $(\mbox{PC})(\tilde{u}_0, s_0, b)$. In the rest of this section, we assume (A1)-(A3). For $T>0$ and $l>0$ such that $s_0<l$ we set
\begin{align*}
& M(T, s_0, l):=\{ s\in W^{1,2}(0, T) | s_0 \leq s \leq l \mbox{ on } [0, T], s(0)=s_0\}.
\end{align*}
Also, for given $s\in M(T, s_0, l)$, we define two solution mappings as follows: $\Psi : M(T, s_0, l) \to  W^{1,2}(0, T; L^2(0, 1))\cap L^{\infty}(0, T; H^1(0, 1))$ by
$\Psi (s)=\tilde{u}$, where $\tilde{u}$ is a unique solution of $(\mbox{AP})(\tilde{u}_0, s, b)$, and $\Gamma_{T}: M(T, s_0, l) \to W^{1,2}(0, T)$ by
$\Gamma_{T}(s)=s_0 + \int_0^t a_0 \sigma(\Psi(s)(\tau, 1)) d\tau$ for $t\in [0, T]$. Moreover, for any $K>0$ we put
\begin{align*}
& M_K(T):=\{s\in M(T, s_0, l) | \ |s|_{W^{1,2}(0, T)}\leq K \}.
\end{align*}

%
%
%
%

Now, we show that for some $T>0$, $\Gamma_{T}$ is a contraction mapping on the closed set of $M_K(T)$ for any $K>0$.
\begin{lem}
\label{lem6}
Let $K>0$. Then, there exists a positive constant $T^*\leq T$ such that 
the mapping $\Gamma_{T^*}$ is a contraction on the closed set $M_K(T^*)$ in $W^{1,2}(0, T^*)$.
\end{lem}

\begin{proof}
For $T>0$ and $l>0$ such that $s_0<l$, let $s\in M(T, s_0, l)$ and $\tilde{u}=\Psi(s)$. First, we note that it holds 
\begin{align}
& |\Psi(s)|_{W^{1,2}(0, T;L^2(0, 1))} + |\Psi(s)|_{L^{\infty}(0, T;H^1(0, 1))} \leq C \mbox{ for }s\in M_K(T),
\label{4-1}
\end{align}
where $C=C(T,\tilde{u}_0, K, l, b^*, \beta,s_0)$ is a positive constant depending on $T$, $\tilde{u}_0$, $K$, $l$, $b^*$, $\beta$ and $s_0$.

Next, we show that there exists $T_0\leq T$ such that $\Gamma_{T_0}:M_K(T_0)\to M_K(T_0)$ is well-defined. Let $K>0$ and $s\in M_K(T)$. First, by the definition of $\sigma$, we see that 
\begin{align}
\Gamma_{T}(s)(t) &= s_0 + \int_0^t a_0 \sigma(\Psi(s)(\tau, 1)) d\tau \geq s_0 \mbox{ for }t\in [0, T].
\label{4-2}
\end{align}
Also, by (\ref{1d}) and (\ref{4-1}), we have that $|\tilde{u}(t, 1)| \leq \sqrt{C_e}|\tilde{u}(t)|_{H^1(0, 1)} \leq \sqrt{C_e}C$ for a.e. $t\in (0, T)$. Hence, by $\sigma(r)\leq r$ for $r\in \mathbb{R}$ we obtain that 
\begin{align}
\Gamma_{T}(s)(t) \leq s_0 +a_0\sqrt{C_e} C T, \quad 
\int_0^t |\Gamma_T(s)(\tau)|^2 d\tau
\leq 2s^2_0T + 2a^2_0 T^3 (C_e C^2), 
\label{4-4}
\end{align}
and
\begin{align}
& \int_0^t |\Gamma'_T(s)(\tau)|^2d\tau \leq a^2_0\int_0^t|\Psi(s)(\tau, 1)|^2 d\tau 
\leq a^2_0 T C_e C^2.
\label{4-5}
\end{align}
Therefore,  by (\ref{4-4}) and (\ref{4-5}) we see that there exists $T_0\leq T$ such that $\Gamma_{T_0}(s)\in M_K(T_0)$.

Next, for $s_1$ and $s_2\in M_K(T_0)$, let $\tilde{u}_1=\Psi(s_1)$ and $\tilde{u}_2=\Psi(s_2)$ and set $\tilde{u}=\tilde{u}_1-\tilde{u}_2$, $s=s_1-s_2$.
Then, it holds that
\begin{align}
& \frac{1}{2}\frac{d}{dt}|\tilde{u}(t)|^2_H - \int_0^1 \left( \frac{1}{s^2_1(t)}\tilde{u}_{1yy}(t)-\frac{1}{s^2_2(t)}\tilde{u}_{2yy}(t)\right)\tilde{u}(t)dy \nonumber \\
&=\int_0^1 \left(\frac{y s_{1t}(t)}{s_1(t)} \tilde{u}_{1y}(t)-\frac{ys_{2t}(t)}{s_2(t)} \tilde{u}_{2y}(t)\right)\tilde{u}(t)dy.
\label{4-6}
\end{align}
For the second term of the left-hand side of (\ref{4-6}), we observe that 
\begin{align*}
& - \int_0^1 \left( \frac{1}{s^2_1(t)}\tilde{u}_{1yy}(t)-\frac{1}{s^2_2(t)}\tilde{u}_{2yy}(t)\right)\tilde{u}(t)dy\\
=& \int_0^1 \left(\frac{1}{s^2_1(t)}\tilde{u}_{1y}(t)-\frac{1}{s^2_2(t)}\tilde{u}_{2y}(t)\right) \tilde{u}_y(t) dy\\
& - \left (\frac{1}{s^2_1(t)}\tilde{u}_{1y}(t, 1)-\frac{1}{s^2_2(t)}\tilde{u}_{2y}(t, 1)\right)\tilde{u}(t, 1)
+ \left (\frac{1}{s^2_1(t)}\tilde{u}_{1y}(t, 0)-\frac{1}{s^2_2(t)}\tilde{u}_{2y}(t, 0)\right)\tilde{u}(t, 0)\\
=:& I_1+I_2+I_3.
\end{align*}
For the term $I_1$, the following estimate below holds:
\begin{align*}
& I_1= \frac{1}{s^2_1(t)}|\tilde{u}_{y}(t)|^2_{L^2(0, 1)}+ \int_0^1 \left(\frac{1}{s^2_1(t)}-\frac{1}{s^2_2(t)}\right)\tilde{u}_{2y}(t)\tilde{u}_y (t)dy\\
\geq & \frac{1}{s^2_1(t)}|\tilde{u}_{y}(t)|^2_{L^2(0, 1)} - \frac{2l|s(t)|}{s^3_0 s_1(t)}|\tilde{u}_{2y}(t)|_{L^2(0, 1)}|\tilde{u}_y(t)|_{L^2(0, 1)} \\
\geq & \left(1-\frac{\eta}{2}\right)\frac{1}{s^2_1(t)}|\tilde{u}_{y}(t)|^2_{L^2(0, 1)} -\frac{1}{2\eta}\left(\frac{2l}{s^3_0}\right)^2|s(t)|^2|\tilde{u}_{2y}(t)|^2_{L^2(0, 1)},
\end{align*}
where $\eta $ is arbitrary positive number. For $I_2$, we separate in the following way:
\begin{align*}
& -\left (\frac{1}{s^2_1(t)}\tilde{u}_{1y}(t, 1)-\frac{1}{s^2_2(t)}\tilde{u}_{2y}(t, 1)\right)\tilde{u}(t, 1)\\
=&a_0\hspace{-1mm}\left(\frac{\tilde{u}_1(t, 1)\sigma(\tilde{u}_1(t, 1))}{s_1(t)}-\frac{\tilde{u}_2(t, 1)\sigma(\tilde{u}_2(t, 1))}{s_2(t)}\right)\hspace{-1mm}\tilde{u}(t, 1)\\
=& a_0\left(\frac{1}{s_1(t)}\biggl(\tilde{u}_1(t, 1)\sigma(\tilde{u}_1(t, 1))-\tilde{u}_2(t, 1)\sigma(\tilde{u}_2(t, 1))\biggr)
 + \left(\frac{1}{s_1(t)}-\frac{1}{s_2(t)}\right)\tilde{u}_2(t, 1)\sigma(\tilde{u}_2(t, 1)) \right) \tilde{u}(t, 1)\\
=&: I_{21}+I_{22}.
\end{align*}
Similarly to (\ref{3-10}), the term $I_{21}$ is non-positive because the function $r\sigma(r)$ is monotone for $r\in \mathbb{R}$. 
For $I_{22}$, using the fact that $\sigma(r)\leq |r|$ for $r\in \mathbb{R}$ and (\ref{1d}), the following inequalities hold:
\begin{align}
\label{4-7}
|I_{22}| & = \left (\frac{s(t)}{s_1(t) s_2(t)} \right) a_0 \tilde{u}_2(t, 1) \sigma(\tilde{u}_2(t, 1)) \tilde{u}(t, 1) \nonumber \\
& \leq \frac{C_e (a_0\tilde{u}^2_2(t, 1))^2}{2s^2_0s^2_1(t)}|\tilde{u}(t)|_{H^1(0, 1)}|\tilde{u}(t)|_{L^2(0, 1)} + \frac{1}{2}|s(t)|^2. 
\end{align}
Put $L^{(1)}_{s_2}(t) = C_e a^2_0|\tilde{u}_2(t, 1)|^4/2s^2_0 $. As for $I_2$, we consider the term $I_3$ as follows:
\begin{align*}
& \left (\frac{1}{s^2_1(t)}\tilde{u}_{1y}(t, 0)-\frac{1}{s^2_2(t)}\tilde{u}_{2y}(t, 0)\right)\tilde{u}(t, 0)\\
=&-\left(\frac{1}{s_1(t)}\beta(b(t)-\gamma \tilde{u}_1(t, 0))-\frac{1}{s_2(t)} \beta(b(t)-\gamma \tilde{u}_2(t, 0))\right)\tilde{u}(t, 0)\\
= & -\frac{1}{s_1(t)}\biggl(\beta(b(t)-\gamma \tilde{u}_1(t, 0))-\beta(b(t)-\gamma \tilde{u}_2(t, 0))\biggr)\tilde{u}(t, 0)\\
& -\left(\frac{1}{s_1(t)}-\frac{1}{s_2(t)}\right)\beta(b(t)-\gamma \tilde{u}_2(t, 0))\tilde{u}(t, 0). 
\end{align*}
Then, by using (\ref{1d}) and (A3), we have that
\begin{align}
|I_3| 
& \leq  \frac{\beta C_e\gamma }{s_1(t)} |\tilde{u}(t)|_{H^1(0, 1)}|\tilde{u}(t)|_{L^2(0, 1)} \nonumber \\
& + \frac{(\beta(b^*+\gamma |\tilde{u}_2(t, 0)|)^2 C_e}{2s^2_0s^2_1(t)}|\tilde{u}(t)|_{H^1(0, 1)}|\tilde{u}(t)|_{L^2(0, 1)} + \frac{1}{2}|s(t)|^2 \mbox{ for }t\in [0, T_0].
\label{4-8}
\end{align}
For the right-hand side of (\ref{4-6}), we can write as follows:
\begin{align*}
& \int_0^1 \left(\frac{y s_{1t}(t)}{s_1(t)} \tilde{u}_{1y}(t)-\frac{y s_{2t}(t)}{s_2(t)} \tilde{u}_{2y}(t)\right)\tilde{u}(t)dy\\
=& \int_0^1 \frac{y s_{1t}(t)}{s_1(t)} \tilde{u}_y(t)\tilde{u}(t)dy+\int_0^1 \frac{ys_t(t)}{s_1(t)} \tilde{u}_{2y}(t)\tilde{u}(t)dy +\int_0^1 \left(\frac{1}{s_1(t)} -\frac{1}{s_2(t)} \right)ys_{2t}(t)\tilde{u}_{2y}(t)\tilde{u}(t)dy \\
&:=I_{41}+I_{42} + I_{43}.
\end{align*}
The three terms are estimated in the following way:
\begin{align*}
I_{41} &\leq \frac{\eta }{2s^2_1(t)}|\tilde{u}_y(t)|^2_{L^2(0, 1)} + \frac{1}{2\eta }|s_{1t}(t)|^2|\tilde{u}(t)|^2_{L^2(0, 1)}, \\
I_{42} &\leq \frac{1}{2s_0}\biggl(|s_t(t)|^2+|\tilde{u}_{2y}(t)|^2_{L^2(0, 1)}|\tilde{u}(t)|^2_{L^2(0, 1)}\biggr), \\
I_{43} &\leq \frac{1}{2s^2_0}\biggl(|s(t)|^2|\tilde{u}_{2y}(t)|^2_{L^2(0, 1)}+|s_{2t}(t)|^2|\tilde{u}(t)|^2_{L^2(0, 1)}\biggr).
\end{align*}
Then, by (\ref{4-6})-(\ref{4-8}) we obtain that 
\begin{align}
& \frac{1}{2}\frac{d}{dt}|\tilde{u}(t)|^2_{L^2(0, 1)}+(1-\eta)\frac{1}{s^2_1(t)}|\tilde{u}_y(t)|^2_{L^2(0, 1)} \nonumber \\
\leq & \frac{\beta C_e\gamma }{s_1(t)} |\tilde{u}(t)|_{H^1(0, 1)}|\tilde{u}(t)|_{L^2(0, 1)} \nonumber \\
& + \frac{1}{s^2_1(t)}\biggl(L^{(1)}_{s_2}(t) + \frac{(\beta (b^*+\gamma |\tilde{u}_2(t, 0)|)^2C_e}{2s^2_0}\biggr) |\tilde{u}(t)|_{H^1(0, 1)}|\tilde{u}(t)|_{L^2(0, 1)} \nonumber \\
& + \left(\frac{1}{2\eta }|s_{1t}(t)|^2 +  \frac{1}{2s_0}|\tilde{u}_{2y}(t)|^2_{L^2(0, 1)} + \frac{1}{2s^2_0}|s_{2t}(t)|^2 \right)|\tilde{u}(t)|^2_{L^2(0, 1)} \nonumber \\
& + \left( \frac{1}{2s^2_0}|\tilde{u}_{2y}(t)|^2_{L^2(0, 1)} + \frac{1}{2\eta}\left(\frac{2l}{s^3_0}\right)^2|\tilde{u}_{2y}|^2_{L^2(0, 1)} + 1 \right)|s(t)|^2 + \frac{1}{2s_0}|s_t(t)|^2.
\label{4-9}
\end{align}
Here, by (\ref{1d}) and (\ref{4-1}), we see that
\begin{align}
|\tilde{u}_i(t, 0)|^2 & \leq C_e (|\tilde{u}_{iy}(t)|_{L^2(0, 1)}|\tilde{u}_i(t)|_{L^2(0, 1)} + |\tilde{u}_i(t)|^2_{L^2(0, 1)}) \nonumber \\
& \leq 2C_eC^2 \mbox{ for }t\in [0, T_0],
\label{4-10}
\end{align}
where $C$ is the same constant as in (\ref{4-1}). Then, by (\ref{4-10}) we note that $\{ L^{(1)}_{s_2} | s_2\in M_k(T)\}$ is bounded in $L^{\infty}(0, T_0)$. Also, by putting $C_5=(\beta (b^*+2\gamma C_eC^2))^2C_e$ and Young's inequality it follows that 
\begin{align*}
& \frac{\beta C_e\gamma }{s_1(t)}|\tilde{u}(t)|_{H^1(0, 1)}|\tilde{u}(t)|_{L^2(0, 1)}\\
\leq & \frac{\beta C_e\gamma }{s_1(t)}\biggl(|\tilde{u}_y(t)|_{L^2(0, 1)}|\tilde{u}(t)|_{L^2(0, 1)} + |\tilde{u}(t)|^2_{L^2(0, 1)}\biggr)\\
\leq & \beta C_e \gamma \left(\frac{\eta}{2s^2_1(t)}|\tilde{u}_y(t)|^2_{L^2(0, 1)} + (\frac{1}{2\eta } + \frac{1}{s_0})|\tilde{u}(t)|^2_{L^2(0, 1)} \right),
\end{align*}
and
\begin{align*}
& \left(L^{(1)}_{s_2}(t) + \frac{C_5}{2s^2_0}\right) \frac{1}{s^2_1(t)}|\tilde{u}(t)|_{H^1(0, 1)}|\tilde{u}(t)|_{L^2(0, 1)} \\
\leq & \left(L^{(1)}_{s_2}(t) + \frac{C_5}{2s^2_0}\right) \frac{1}{s^2_1(t)}(|\tilde{u}_y(t)|_{L^2(0, 1)}|\tilde{u}(t)|_{L^2(0, 1)} + |\tilde{u}(t)|^2_{L^2(0, 1)})\\
\leq &  \left(L^{(1)}_{s_2}(t) + \frac{C_5}{2s^2_0}\right) \left [\frac{1}{s^2_1(t)}\frac{\eta }{2}|\tilde{u}_y(t)|^2_{L^2(0, 1)} + \frac
{1}{s^2_0}(\frac{1}{2\eta } + 1)|\tilde{u}(t)|^2_{L^2(0, 1)} \right]. 
\end{align*}
Accordingly, by applying these results to (4.8) and taking a suitable $\eta=\eta_0$, we have
\begin{align}
& \frac{1}{2}\frac{d}{dt}|\tilde{u}(t)|^2_{L^2(0, 1)}+\frac{1}{2}\frac{1}{s^2_1(t)}|\tilde{u}_y(t)|^2_{L^2(0, 1)} \nonumber \\
\leq &  \beta C_e \gamma \left(\frac{1}{2\eta_0 } + \frac{1}{s_0} \right)|\tilde{u}(t)|^2_{L^2(0, 1)} \nonumber \\
& + \left(L^{(1)}_{s_2}(t) + \frac{C_5}{2s^2_0}\right)\frac{1}{s^2_0}\left(\frac{1}{2\eta_0 } + 1\right)|\tilde{u}(t)|^2_{L^2(0, 1)} \nonumber \\
& + \left(\frac{1}{2\eta_0 }|s_{1t}(t)|^2 +  \frac{1}{2s_0}|\tilde{u}_{2y}(t)|^2_{L^2(0, 1)} + \frac{1}{2s^2_0}|s_{2t}(t)|^2 \right)|\tilde{u}(t)|^2_{L^2(0, 1)} \nonumber \\
& + \left(\frac{1}{2s^2_0}|\tilde{u}_{2y}(t)|^2_{L^2(0, 1)} + \frac{1}{2\eta_0}\left(\frac{2l}{s^3_0}\right)^2|\tilde{u}_{2y}(t)|^2_{L^2(0, 1)}+ 1 \right)|s(t)|^2 + \frac{1}{2s_0}|s_t(t)|^2.
\label{4-11}
\end{align}
Now, we put the summation of all coefficients of $|\tilde{u}(t)|^2_{L^2(0, 1)}$ by $L^{(2)}_{s}(t)$ for $t\in [0, T_0]$ and 
$L^{(3)}_{s_2}(t)=|\tilde{u}_{2y}(t)|^2_{L^2(0, 1)}/2s^2_0 + (4l^2|\tilde{u}_{2y}(t)|^2_{L^2(0, 1)})/2\eta_0s^6_0+ 1+ 1/2s_0$. Then, we have
\begin{align}
& \frac{1}{2}\frac{d}{dt}|\tilde{u}(t)|^2_{L^2(0, 1)}+\frac{1}{2}\frac{1}{s^2_1(t)}|\tilde{u}_y(\tau)|^2_{L^2(0, 1)} \nonumber \\
\leq & L^{(2)}_{s}(t)|\tilde{u}(t)|^2_{L^2(0, 1)} + L^{(3)}_{s_2}(t)(|s(t)|^2 + |s_t(t)|^2) \mbox{ for }t\in [0, T_0].
\label{4-12} 
\end{align}
Here, using (\ref{4-1}) and the fact that $s_i\in M_K(T_0)$ for $i=1, 2$, we see that $L^{(2)}_s\in L^1(0, T_0)$ and $L^{(3)}_{s_2}\in L^{\infty}(0, T_0)$.
Therefore, Gronwall's inequality guarantees that
\begin{align}
& \frac{1}{2}|\tilde{u}(t)|^2_{L^2(0, 1)}+\frac{1}{2}\frac{1}{s^2_1(t)}\int_0^t|\tilde{u}_y(\tau)|^2_{L^2(0, 1)} d\tau \nonumber \\
\leq & \left( |L^{(3)}_{s_2}|_{L^{\infty}(0, T_0)}|s|^2_{W^{1,2}(0, T)} \right) e^{2\int_0^t L^{(2)}_{s}(\tau)d\tau} \mbox{ for }t\in [0, T_0].
\label{4-13}
\end{align}
By using (\ref{4-13}) we show that there exists $T^*\leq T_0$ such that $\Gamma_{T^*}$ is a contraction mapping on the closed subset of $M_K(T^*)$.
To do so, from the subtraction of the time derivatives of $\Gamma_{T_0}(s_1)$ and $\Gamma_{T_0}(s_2)$ and relying on (\ref{1d}) and (\ref{4-13}), we have for $T_1\leq T_0$ the following estimate:
\begin{align}
& |(\Gamma_{T_1}(s_1))_t-(\Gamma_{T_1}(s_2))_t|_{L^2(0, T_1)} \nonumber \\
\leq & a_0 \biggl(|\sigma(\tilde{u}_1(\cdot, 1))-\sigma(\tilde{u}_2(\cdot, 1))|_{L^2(0, T_1)} \biggr) \nonumber \\
\leq & a_0\sqrt{C_e}\biggl(\int_0^{T_1}(|\tilde{u}_y(t)|_{L^2(0, 1)} |\tilde{u}(t)|_{L^2(0, 1)} + |\tilde{u}(t)|^2_{L^2(0, 1)})dt \biggr)^{1/2} \nonumber \\
\leq & a_0 \sqrt{C_e} \left(|\tilde{u}|^{\frac{1}{2}}_{L^{\infty}(0, T;L^2(0, 1))}\left(\int_0^{T_1}|\tilde{u}_y(t)|_{L^2(0, 1)} dt \right)^{\frac{1}{2}} + \sqrt{T_1} |\tilde{u}|_{L^{\infty}(0, T; L^2(0, 1))} \right).  
\label{4-14}
\end{align}
Using (\ref{4-13}), we obtain
\begin{align}
|\Gamma_{T_1}(s_1)-\Gamma_{T_1}(s_2)|_{W^{1,2}(0, T_1)} \leq T_1C_6 \biggl (T^{\frac{1}{4}}_1|s|_{W^{1,2}(0, T_1)}+ \sqrt{T_1}|s|_{W^{1,2}(0, T_1)}\biggr),
\label{4-14-1}
\end{align}
where $C_6$ is a positive constant obtained by (\ref{4-13}). Therefore, by (\ref{4-14}) and (\ref{4-14-1}) we see that there exists $T^*\leq T_0$ such that $\Gamma_{T^*}$ is a contraction mapping on a closed subset of $M_K(T^*)$.
\end{proof}

From Lemma \ref{lem6}, by applying Banach's fixed point theorem, there exists $s\in M_K(T^*)$, where $T^*$ is the same as in Lemma \ref{lem6} such that
$\Gamma_{T^*}(s)=s$. This implies that $(\mbox{PC})(\tilde{u}_0, s_0, b)$ has a unique solution $(s, \tilde{u})$ on $[0, T^*]$. Thus, we can prove Theorem \ref{t2}. Moreover, this shows that 
by the change of variables (\ref{2-7})
a pair of the function $(s, u)$ is a solution of $(\mbox{P})(u_0, s_0, b)$ on $[0, T^*]$.

At the end of this section, we show the boundedness of  the solution to $(\mbox{P})(u_0, s_0, b)$. 

%
%
%
%

\begin{lem}
\label{lem7}
Let $T>0$ and $(s, u)$ be a solution of $(\mbox{P})(u_0, s_0, b)$ on $[0, T]$. Then, $0 \leq u(t)\leq b^*/\gamma$ on $[0, s(t)]$ for $t\in [0, T]$.
\end{lem}

\begin{proof}
First, we show that $u(t) \geq 0$ on $[0, s(t)]$ for $t\in [0, T]$. From (1.1), we have that 
\begin{align}
& \frac{1}{2}\frac{d}{dt}\int_0^{s(t)}|[-u(t)]^{+}|^2 dz -\frac{s_t(t)}{2}|[-u(t, s(t))]^+|^2 \nonumber \\
& + \int_0^{s(t)} u_{zz}(t)[-u(t)]^+ dz =0 \mbox{ for a.e. }t\in [0, T].
\label{4-15}
\end{align}
By the boundary conditions (1.2) and (1.3) it follows that 
\begin{align*}
u_z(t, s(t))[-u(t, s(t))]^+ & = -u(t, s(t))s_t(t)[-u(t, s(t))]^+ = s_t(t)|[-u(t, s(t))]^+|^2
\end{align*}
and
\begin{align*}
&-u_z(t, 0)[-u(t, 0)]^+ =\beta(b(t)-\gamma u(t, 0))[-u(t, 0)]^+\geq 0.
\end{align*}
Therefore, we derive that
\begin{align}
\frac{d}{dt}\int_0^{s(t)}|[-u(t)]^{+}|^2 dz + \frac{s_t(t)}{2}|[-u(t, s(t))]^+|^2 + \int_0^{s(t)}|[-u(t)]_z^+|^2 dz \leq 0 \mbox{ for a.e. }t\in [0, T].
\label{4-16}
\end{align}
Note that by $s_t(t)=a_0\sigma(u(t, s(t))$, the second term in the left-hand side of (\ref{4-16}) is equal to 0. 
Therefore, by integrating (\ref{4-16}) over $[0, T]$ we conclude that $u \geq 0$ on $[0, s(t)]$ for $t\in [0, T]$. 

Next, we show that $u(t) \leq b^*/\gamma $ on $[0, s(t)]$ for $t\in [0, T]$.
Put $U(t, z)=[u(t, z)-b^*/\gamma]^+$ for $z\in [0, s(t)]$ and $t\in [0, T]$. Then, we have that
\begin{align}
\frac{1}{2}\frac{d}{dt}\int_0^{s(t)}|U(t, z)|^2 dz -\frac{s_t(t)}{2}|U(t, s(t))|^2 -\int_0^{s(t)} u_{zz}(t)U(t, z)dz =0 \mbox{ for a.e. }t\in[0, T].
\label{4-20}
\end{align}
Using the boundary condition (\ref{1-2}), it holds that 
\begin{align*}
& -u_z(t, s(t))U(t, s(t)) = u(t, s(t))s_t(t) U(t, s(t)) \\
= & s_t(t)\left(u(t, s(t))-\frac{b^*}{\gamma}\right)U(t, s(t)) + s_t(t) \frac{b^*}{\gamma} U(t, s(t)) \\  
=& s_t(t)|U(t, s(t))|^2 + s_t(t) \frac{b^*}{\gamma}U(t, s(t)).
\end{align*}
Also, by (\ref{1-3}) and $b\leq b^*$, we observe that 
\begin{align*}
& u_z(t, 0)U(t, 0) =  -\beta(b(t)-\gamma u(t, 0))U(t, 0)\\
= & \beta (\gamma u(t, 0)-b^* + b^*-b(t))U(t, 0) \\
= & \beta \gamma |U(t, 0)|^2 + \beta(b^*-b(t))U(t, 0) \geq 0.
\end{align*}
By applying the above two results to (\ref{4-20}) we obtain that 
\begin{align}
& \frac{1}{2}\frac{d}{dt}\int_0^{s(t)}|U(t, z)|^2 dz + \int_0^{s(t)} |U_z(t, z)|^2 dz \nonumber \\
& + \frac{s_t(t)}{2}|U(t, s(t))|^2 + s_t(t)\frac{b^*}{\gamma}U(t, s(t)) \leq 0 \mbox{ for a.e. }t\in[0, T].
\label{4-21}
\end{align}
Here, by $s_t(t)=a_0\sigma(u(t, s(t)))$ we notice that $s_t(t) \geq 0$ on $[0, T]$, and the third and forth terms in the left-hand side of (\ref{4-21}) are non-negative. Therefore, we have that 
\begin{align}
\frac{1}{2}\frac{d}{dt}\int_0^{s(t)}|U(t, z)|^2 dz + \int_0^{s(t)} |U_z(t, z)|^2 dz \leq 0 \mbox{ for a.e. }t\in[0, T].
\label{4-22}
\end{align}
Finally,  by integrating (\ref{4-22}) over $[0, t]$ for $t\in [0, T]$ and using (A3), we see that $u(t) \leq b^*/\gamma$ on $[0, s(t)]$ for $t\in [0, T]$.
Thus, Lemma \ref{lem7} is proven.
\end{proof}

By Lemma \ref{lem7}, we can conclude that Theorem \ref{t1} holds. 

%
%
%
%
%
%

\section{Proof of Theorem \ref{t3}}
\label{GE}

In this section, we prove Theorem \ref{t3} which ensure the existence and uniqueness of a globally-in-time solution of (P)$(u_0, s_0, b)$. 
First, we provide uniform estimates of a solution of (P)$(u_0, s_0, b)$. 
\begin{lem}
\label{lem8}
Let $(s, u)$ be a solution of (P)$(u_0, s_0, b)$ on $[0, T]$ satisfying $0 \leq u \leq b^*/\gamma$ on $[0, s(t)]$ for $t\in [0, T]$. Then, there exists a positive constant $\tilde{C}$ which is independent of $T$ such that 
\begin{align}
\label{5-1}
& \int_{0}^t|u_t(\tau)|^2_{L^2(0, s(\tau))}d\tau + |u_z(t)|^2_{L^2(0, s(t))} \leq \tilde{C}\mbox{ for all }t\in (0, T). 
\end{align}
\end{lem}
\begin{proof}
Let $(s, u)$ be a solution of (P)$(u_0, s_0, b)$ on $[0, T]$ such that $0\leq u \leq b^* /\gamma$ on $Q_s(T)$. Then, by the change of variables (\ref{2-0}) 
we see that $(s, \tilde{u}$) is a solution of (PC)$(\tilde{u}_0, s_0, b)$ on $[0, T]$ in the sense of Definition \ref{def1} and satisfies that $0\leq \tilde{u} \leq b^*/\gamma $ on $Q(T)$. Now, we put $v_h(t)=\frac{\tilde{u}(t)-\tilde{u}(t-h)}{h}$ for $h>0$ and $u(t)=u(0)=u_0$ and $b(t)=b(0)$ for $t<0$. 
By (\ref{1-1}),  it holds that 
\begin{align}
\label{5-2}
\int_0^1 \tilde{u}_t(t) s(t)v_h(t) dy -\int_0^1\frac{1}{s(t)}\tilde{u}_{yy}(t)v_h(t) dy =\int_0^1 \frac{ys_t(t)\tilde{u}_y(t)}{s(t)} s(t) v_h(t) dy \quad \mbox{ for }t\in [0, T]. 
\end{align}
Then, using (\ref{1-2})-(\ref{1-4}) and $s_t(t)=a_0\sigma(\tilde{u}(t, 1))=a_0\tilde{u}(t, 1)$ for $t\in [0, T]$, we observe that 
\begin{align}
\label{5-3}
& -\int_0^1\frac{1}{s(t)}\tilde{u}_{yy}(t) \frac{\tilde{u}(t)-\tilde{u}(t-h)}{h} dy \nonumber \\
= & a_0 \tilde{u}^2(t, 1)v_h(t, 1) -\beta(b(t)-\gamma \tilde{u}(t, 0)) v_h(t, 0) + \int_0^1 \frac{1}{s(t)}\tilde{u}_y(t)v_{hy}(t) dy, 
\end{align}
and 
\begin{align}
\label{5-4}
& \int_0^1 \frac{1}{s(t)}\tilde{u}_y(t)v_{hy}(t) dy \nonumber \\
\geq & \frac{1}{2h}\int_0^1 \frac{1}{s(t)}(|\tilde{u}_y(t)|^2-|\tilde{u}_y(t-h)|^2)dy \nonumber \\
= & \frac{1}{2h}\biggl[ \int_0^{s(t)}|u_z(t)|^2 dz - \int_0^{s(t-h)}\frac{s(t-h)}{s(t)}|u_z(t-h)|^2 dz\biggr] \nonumber \\
= & \frac{1}{2h}\biggl[ \int_0^{s(t)}|u_z(t)|^2 dz - \int_0^{s(t-h)}|u_z(t-h)|^2 dz + \int_0^{s(t-h)}\frac{s(t)-s(t-h)}{s(t)}|u_z(t-h)|^2 dz\biggr].
\end{align}
Here, for $t\in [0, T]$ the following inequality holds:
\begin{align}
\label{5-5}
a_0 \tilde{u}^2(t, 1)v_h(t, 1) &= a_0 \frac{\tilde{u}^3(t, 1)-\tilde{u}^2(t, 1)\tilde{u}(t-h, 1)}{h} \nonumber \\
& \geq a_0 \frac{\tilde{u}^3(t, 1)-\tilde{u}^3(t-h, 1)}{3h}.
\end{align}
Also, by introducing $\Phi(b(t), r)=-\beta(b(t)r-\frac{\gamma }{2}r^2)$ for $r\in \mathbb{R}$, it is easy to see that $\frac{\partial^2}{\partial r^2}\Phi(b(t), r)=\beta \gamma \geq 0$ for $r\in \mathbb{R}$. Hence, for $t\in [0, T]$, $\Phi(b(t), \tilde{u}(t, 0))$ is convex with respect to the second component so that we can see that the following inequality holds. 
\begin{align}
\label{5-6}
-\beta(b(t)-\gamma \tilde{u}(t, 0))v_h(t, 0) \geq \frac{\Phi(b(t), \tilde{u}(t, 0))-\Phi(b(t), \tilde{u}(t-h, 0))}{h} 
\mbox{ for }t\in [0, T].
\end{align}
Combining (\ref{5-2})-(\ref{5-6}) with (\ref{5-1}) , we have 
\begin{align}
\label{5-7} 
& \int_0^1\tilde{u}_t(t) s(t) v_h(t)dy \nonumber \\
& + \frac{1}{2h}\biggl[ \int_0^{s(t)}|u_z(t)|^2 dz - \int_0^{s(t-h)}|u_z(t-h)|^2 dz + \int_0^{s(t-h)}\frac{s(t)-s(t-h)}{s(t)}|u_z(t-h)|^2 dz\biggr] \nonumber \\
& + a_0 \frac{\tilde{u}^3(t, 1)-\tilde{u}^3(t-h, 1)}{3h} + \frac{\Phi(b(t), \tilde{u}(t, 0))-\Phi(b(t), \tilde{u}(t-h, 0))}{h} \nonumber \\
\leq & \int_0^1 ys_t(t)\tilde{u}_y(t) v_h(t) dy \mbox{ for }t\in [0, T].
\end{align}
Now, we integrate (\ref{5-7}) over $[0, t_1]$ for $t_1\in (0 ,T]$ and take the limit as $h\to 0$. Then, by the change of variables (\ref{2-7}) the first term of the left-hand side of (\ref{5-7}) is as follows:
\begin{align}
\label{5-8}
& \lim_{h\to 0}\int_0^{t_1} \int_0^1 \tilde{u}_t(t) s(t) v_h(t) dy dt = \int_0^{t_1}\int_0^1 |\tilde{u}_t(t)|^2 s(t) dy dt \nonumber \\
=& \int_0^{t_1}\int_0^{s(t)} \biggl(|u_t(t)|^2 +2u_t(t)u_z(t)\frac{z}{s(t)}s_t(t) + \left(u_z(t)\frac{z}{s(t)}s_t(t)\right)^2 \biggr) dzdt.
\end{align}
As arguing the local existence, the function $t\to \int_0^{s(t)}|u_z(t)|^2dz$ is absolutely continuous on $[0, T]$. Then, the second and third terms of  the left-hand side of (\ref{5-7}) can be dealt with as 
\begin{align}
\label{5-9}
& \lim_{h\to0} \frac{1}{2h}\int_0^{t_1}\biggl(\int_0^{s(t)}|u_z(t)|^2dz-\int_0^{s(t-h)}|u_z(t-h)|^2dz\biggr)d\tau \nonumber \\
= & \frac{1}{2}\biggl(\int_0^{s(t_1)}|u_z(t_1)|^2dz - \int_0^{s_0}|u_z(0)|^2dz\biggr), 
\end{align}
and 
\begin{align}
\label{5-10}
\lim_{h\to 0}\frac{1}{2}  \int_0^{t_1}\int_0^{s(t-h)} \frac{1}{s(t)} \frac{s(t)-s(t-h)}{h}|u_z(t-h)|^2 dzdt = \frac{1}{2} \int_0^{t_1}\int_0^{s(t)} \frac{s_t(t)}{s(t)}|u_z(t)|^2dzdt.
\end{align}
Moreover, since $\tilde{u}$ is continuous on $\overline{Q(T)}$ we have that 
\begin{align}
\label{5-11}
\lim_{h\to 0}\frac{a_0}{3h}\int_0^{t_1}\biggl(\tilde{u}^3(t, 1)-\tilde{u}^3(t-h, 1)\biggr)dt = & \lim_{h\to 0} \biggl(\frac{a_0}{3h}\int_{t_1-h}^{t_1} \tilde{u}^3(t, 1) dt\biggr) -\frac{a_0}{3}\tilde{u}^3_0(1) \nonumber \\
= & \frac{a_0}{3}\tilde{u}^3(t_1, 1)-\frac{a_0}{3}\tilde{u}^3_0(1).
\end{align}
Similarly to the derivation of (\ref{5-11}), 
\begin{align}
\label{5-12}
& \lim_{h\to 0}\frac{1}{h}\int_0^{t_1}\biggl( \Phi(b(t), \tilde{u}(t, 0))-\Phi(b(t), \tilde{u}(t-h, 0)) \biggr) dt \nonumber \\ 
= & \Phi(b(t_1), \tilde{u}(t_1, 0)) - \Phi(b(0), \tilde{u}_0(0)) \nonumber \\ 
& + \lim_{h\to 0}\biggl(-\frac{1}{h}\int_0^{t_1}\biggl[\Phi(b(t), \tilde{u}(t-h, 0))-\Phi(b(t-h), \tilde{u}(t-h, 0))\biggr]dt\biggr). 
\end{align}
For the last term of the right-hand side of (\ref{5-12}) we observe that 
\begin{align}
\label{5-13}
& \lim_{h\to 0}\biggl(-\frac{1}{h}\int_0^{t_1}\biggl[\Phi(b(t), \tilde{u}(t-h, 0))-\Phi(b(t-h), \tilde{u}(t-h, 0))\biggr]dt\biggr) \nonumber \\
\geq & \lim_{h\to 0}\biggl(-\frac{1}{h}\int_0^{t_1} \beta |b(t)-b(t-h)||\tilde{u}(t-h, 0)|dt \biggr) \nonumber \\
\geq & \lim_{h\to 0}\biggl(-\frac{\beta}{h}\int_0^{t_1}\biggl(\int_{t-h}^t |b_t(\tau)|d\tau\biggr)|\tilde{u}(t-h, 0)|dt \biggr) \nonumber \\
\geq & -\frac{\beta b^*}{\gamma} \int_0^{t_1}|b_t(t)| dt.
\end{align}
From (\ref{5-7}) and the estimates (\ref{5-8})-(\ref{5-13}), we obtain that 
\begin{align}
\label{5-14} 
& \int_0^{t_1}\int_0^{s(t)} \biggl(|u_t(t)|^2 +2u_t(t)u_z(t)\frac{z}{s(t)}s_t(t) + \left(u_z(t)\frac{z}{s(t)}s_t(t)\right)^2 \biggr) dzdt  \nonumber \\
& + \frac{1}{2}\int_0^{s(t_1)}|u_z(t_1)|^2dz -\frac{1}{2}\int_0^{s_0}|u_z(0)|^2dz + \frac{1}{2}\int_0^{t_1}\int_0^{s(t)} \frac{s_t(t)}{s(t)}|u_z(t)|^2dzdt \nonumber \\
& +\frac{a_0}{3}\tilde{u}^3(t_1, 1) - \frac{a_0}{3}\tilde{u}^3_0(1) + \Phi(b(t_1), \tilde{u}(t_1, 0))-\Phi(b(0), \tilde{u}_0(0)) - \frac{\beta b^*}{\gamma }\int_0^{t_1}|b_t(t)|dt  \nonumber \\
\leq & \int_0^{t_1} \int_0^{s(t)} \biggl (u_t(t)u_z(t)\frac{z}{s(t)}s_t(t) + \left(u_z(t)\frac{z}{s(t)}s_t(t)\right)^2 \biggr) dzdt \ \mbox{ for }t_1\in [0, T].
\end{align}
Then, we see that the third term of the left-hand side of (\ref{5-14}) is same to the second term of the right-hand side of (\ref{5-14}) . Then, by moving the second term of the left-hand side of (\ref{5-14}) to the right-hand side we have that 
\begin{align}
\label{5-15}
& \int_0^{t_1}\int_0^{s(t)} |u_t(t)|^2 dzdt  \nonumber \\
& + \frac{1}{2}\int_0^{s(t_1)}|u_z(t_1)|^2dz -\frac{1}{2}\int_0^{s_0}|u_z(0)|^2dz + \frac{1}{2}\int_0^{t_1}\int_0^{s(t)} \frac{s_t(t)}{s(t)}|u_z(t)|^2dzdt \nonumber \\
& + \frac{a_0}{3} \tilde{u}^3(t_1, 1) - \frac{a_0}{3} \tilde{u}^3_0(1) + \Phi(b(t_1), \tilde{u}(t_1, 0))-\Phi(b(0), \tilde{u}_0(0))  - \frac{\beta b^*}{\gamma } \int_0^{t_1}|b_t(t)| dt \nonumber \\
\leq & \int_0^{t_1} \int_0^{s(t)} -u_t(t)u_z(t)\frac{z}{s(t)}s_t(t) dzdt \ \mbox{ for }t_1\in [0, T].
\end{align}
Using (\ref{1-1}) and the fact that $s_t(t)\geq 0$ for $t\in [0, T]$ we obtain the following inequality:
\begin{align}
\label{5-16}
& -\int_0^{t_1} \int_0^{s(t)} u_t(t)u_z(t)\frac{z}{s(t)}s_t(t) dzdt \nonumber \\
= & -\int_0^{t_1} \int_0^{s(t)} u_{zz}(t) u_z(t) \frac{z}{s(t)}s_t(t) dzdt \nonumber \\
= & -\int_0^{t_1} \int_0^{s(t)} \frac{1}{2}\biggl(\frac{\partial}{\partial z}|u_z(t)|^2\biggr) \frac{z}{s(t)}s_t(t) dzdt \nonumber \\
= & -\int_0^{t_1} \frac{1}{2}|u_z(t, s(t))|^2s_t(t) dt + \frac{1}{2}\int_0^{t_1}\int_0^{s(t)}\frac{s_t(t)}{s(t)}|u_z(t)|^2dzdt. \nonumber \\
\leq & \frac{1}{2}\int_0^{t_1}\int_0^{s(t)}\frac{s_t(t)}{s(t)}|u_z(t)|^2dzdt. 
\end{align}
Hence, by (\ref{5-15}) and (\ref{5-16}) we have that 
\begin{align}
\label{5-17}
& \int_0^{t_1}\int_0^{s(t)} |u_t(t)|^2 dzdt \nonumber \\ 
& + \frac{1}{2}\int_0^{s(t_1)}|u_z(t_1)|^2dz - \frac{1}{2}\int_0^{s_0}|u_z(0)|^2dz + \frac{1}{2}\int_0^{t_1}\int_0^{s(t)} \frac{s_t(t)}{s(t)}|u_z(t)|^2dzdt \nonumber \\
& + \frac{a_0}{3} \tilde{u}^3(t_1, 1) - \frac{a_0}{3} \tilde{u}^3_0(1) + \Phi(b(t_1), \tilde{u}(t_1, 0))-\Phi(b(0), \tilde{u}_0(0)) - \frac{\beta b^*}{\gamma } \int_0^{t_1}|b_t(t)| dt \nonumber \\ 
\leq & \frac{1}{2}\int_0^{t_1}\int_0^{s(t)}\frac{s_t(t)}{s(t)}|u_z(t)|^2dzdt \ \mbox{ for }t_1\in [0, T].
\end{align}
The forth term in the left-hand side and the right-hand side are canceled out and the fifth and eighth terms in the left-hand side are positive. 
Therefore, we finally obtain that 
\begin{align}
\label{5-18} 
& \int_0^{t_1}\int_0^{s(t)} |u_t(t)|^2 + \frac{1}{2}\int_0^{s(t_1)}|u_z(t_1)|^2dz  \nonumber \\
\leq & \frac{1}{2}\int_0^{s_0}|u_z(0)|^2dz + \frac{a_0}{3} \tilde{u}^3_0(1) -\Phi(b(t_1), \tilde{u}(t_1, 0)) + \frac{\beta b^*}{\gamma } \int_0^{t_1}|b_t(t)| dt \mbox{ for }t_1\in [0, T]. 
\end{align}
In the right-hand side of (\ref{5-18}), by (A2) and $0\leq \tilde{u}(t)\leq b^*/\gamma $ on $[0, 1]$ for $t\in [0, T]$ and the definition of $\Phi$, we can estimate as follows:
\begin{align}
\label{5-21}
-\Phi(b(t_1), \tilde{u}(t_1, 0)) = -\beta b(t_1) \tilde{u}(t_1, 0) + \frac{\beta \gamma}{2} \tilde{u}^2(t_1, 0) \leq \frac{\beta \gamma}{2} \left( \frac{b^*}{\gamma} \right)^2. 
\end{align}
Finally, by (\ref{5-21}), $b\in W^{1,2}(0, T)$ as in (A2) and (A3) we see that there exists $\tilde{C}$ which depends on $b^*$, $a_0$, $\gamma$, $\beta$ such that (\ref{5-1}) holds. Thus, Lemma \ref{lem8} is proved. 
\end{proof}

%
%
%
%

At the end of this section, we prove Theorem \ref{t3}. Let $T>0$. 
By the local existence result there exists $T_1<T$ such that (P)$(u_0, s_0, b)$ has a unique solution $(s, u)$ on $[0, T_1]$ satisfying $0 \leq u \leq b^*/\gamma$ on $Q_s(T_1)$. Then, the pair $(s, \tilde{u})$ with the variable (\ref{2-0}) is a solution of $(\mbox{PC})(\tilde{u}_0, s_0, b)$ satisfying $0 \leq \tilde{u} \leq b^*/\gamma$ on $Q(T_1)$. Let put 
\begin{align*}
\tilde{T}:=\mbox{sup}\{T_1>0 |  (\mbox{PC})(\tilde{u}_0, s_0, b) \mbox{ has a solution }(s, \tilde{u}) \mbox{ on } [0, T_1]\}.
\end{align*}
From the local existence result, we deduce that $\tilde{T}>0$. Now, we assume $\tilde{T}<T$. First, by (\ref{2-4}) and the result that $\tilde{u}(t)\geq 0$ on $[0, 1]$ for $t\in [0, \tilde{T})$ we see that $s_t(t) \geq 0$ for $t\in [0, \tilde{T})$, and therefore $s(t) \geq s_0$ for $t\in [0, \tilde{T})$. Also, by putting $L(t)=a_0\frac{b^*}{\gamma}t +s_0$ for $t\in [0, T]$ we have that 
\begin{align}
\label{5-22}
s(t) &= s_0 + \int_0^{t} a_0 \sigma(\tilde{u}(\tau, 1)) d\tau \nonumber \\
&=s_0 + \int_0^t a_0 \tilde{u}(\tau, 1) d\tau \nonumber \\
&\leq  s_0 + a_0\frac{b^*}{\gamma} \tilde{T}=L(\tilde{T}) <L(T)\mbox{ for }t\in [0, \tilde{T}). 
\end{align}
Next, by using  the change of the variable (\ref{2-0}), it holds that 
\begin{align}
& \int_0^1|\tilde{u}_y(t)|^2dy =\int_0^{s(t)}\frac{1}{s(t)}|u_z(t) s(t)|^2dzdt. \nonumber 
\end{align}
Therefore, from (\ref{5-22}) and Lemma \ref{lem8},  we obtain that %
\begin{align}
\label{5-26}
|\tilde{u}_y(t)|^2_{L^2(0, 1)}\leq L(T) \tilde{C} \mbox{ for all }t<\tilde{T}, 
\end{align}
where $\tilde{C}$ is the same constant as in Lemma \ref{lem8}. 
By (\ref{5-26}) we see that for some $\tilde{u}_{\tilde{T}}\in H^1(0, 1)$, 
$\tilde{u}(t)\to \tilde{u}_{\tilde{T}}$ strongly in $L^2(0, 1)$ and weakly in $H^1(0, 1)$
as $t\to \tilde{T}$ and $0 \leq \tilde{u}_{\tilde{T}} \leq b^*/\gamma$ on $(0, 1)$. Also, by $|s_t(t)|\leq a_0b^*/\gamma$ for $t\in [0, \tilde{T})$, $\{s(t)\}_{t\in [0, \tilde{T})}$ is a Cauchy sequence in $\mathbb{R}$ so that for some $s_{\tilde{T}}\in \mathbb{R}$, $s(t) \to s_{\tilde{T}}$ in $\mathbb{R}$ as $t\to \tilde{T}$. Moreover, by (\ref{5-22}), $s_{\tilde{T}}$ satisfies that $0<s_0\leq s_{\tilde{T}}\leq L(\tilde{T})$. Now, we put $u_{\tilde{T}}(z)=\tilde{u}_{\tilde{T}}(\frac{z}{s_{\tilde{T}}})$ for $z\in [0, s_{\tilde{T}}]$. Then,  we see that $u_{\tilde{T}}\in H^1(0, s_{\tilde{T}})$ and $0\leq u_{\tilde{T}}\leq b^*/\gamma $ on $(0, s_{\tilde{T}})$ and we can consider $(s_{\tilde{T}}, u_{\tilde{T}})$ as a initial data.
Therefore, by repeating the argument of the local existence we can extend a solution beyond $\tilde{T}$. 
This is a contradiction for the definition of $\tilde{T}$ and we have a solution on the whole interval $[0, T]$. 
Thus Theorem \ref{t3} is proved.

%
%
%
%
%
%

\section{Large time behavior of the free boundary}
\label{large}

In this section, we discuss the large time behavior of a solution to (P)$(u_0, s_0, b)$ as $t\to \infty$. First, we assume (A2)' replaced by (A2): 
\newline
(A2)': $b\in W^{1, 2}_{loc}([0, \infty))$, $b_t\in L^1(0, \infty)$, lim$_{t\to \infty} b(t)=b_{\infty}$, $b-b_{\infty}\in L^1(0, \infty)$ and 
$b_*\leq b\leq b^*$ on $(0, \infty)$, where $b_*$ and $b^*$ are positive constants as in (A2). 
\newline
Clearly, we see that $b_*\leq b_{\infty}\leq b^*$. Next, we consider the following stationary problem (P)$_{\infty}$: find a pair $(u_{\infty}, s_{\infty}) \in L^2(0, s_{\infty}) \times \mathbb{R}$ satisfying 
$$
\begin{cases}
-u_{\infty zz}=0 \mbox{ on } (0, s_{\infty}), \\
-u_{\infty z}(0)=\beta(b_{\infty}-\gamma u_{\infty}(0)), \ \ -u_{\infty z}(s_{\infty})=0,\\ 
u_{\infty}(s_{\infty})=0.
\end{cases}
$$
By using the change of variables $\tilde{u}_{\infty}(y)=u_{\infty}(ys_{\infty})$ for $y\in (0, 1)$, (P)$_{\infty}$ can be written in the following problem $(\tilde{\mbox{P}})_{\infty}$:
$$
\begin{cases}
-\displaystyle{\frac{1}{s^2_{\infty}}}\tilde{u}_{\infty yy}=0 \mbox{ on } (0, 1), \\[3mm]
-\displaystyle{\frac{1}{s_{\infty}}}\tilde{u}_{\infty y}(0)=\beta(b_{\infty}-\gamma \tilde{u}_{\infty}(0)), \ \ 
-\displaystyle{\frac{1}{s_{\infty}}}\tilde{u}_{\infty y}(1)=0,\\[2mm] 
\tilde{u}_{\infty}(1)=0.
\end{cases}
$$
The next lemma is concerned with non-existence of a solution $(s_{\infty}, \tilde{u}_{\infty})$ of the problem $(\tilde{\mbox{P}})_{\infty}$. 

\begin{lem}
\label{b-1} 
A solution $(s_{\infty}, \tilde{u}_{\infty})$ of $(\tilde{\mbox{P}})_{\infty}$ satisfying $0<s_{\infty}<+\infty$ and $\tilde{u}_{\infty}\in H^2(0, 1)$ does not exist.
\end{lem}

\begin{proof}
Let $(s_{\infty}, \tilde{u}_{\infty})$ be a solution of $(\tilde{\mbox{P}})_{\infty}$ such that $0<s_{\infty}<+\infty$ and $\tilde{u}_{\infty}\in H^2(0, 1)$. Then, it holds that 
\begin{align*}
-\frac{1}{s_{\infty}}\tilde{u}_{\infty y}(1)+ \frac{1}{s_{\infty}}\tilde{u}_{\infty y}(0)=0.
\end{align*}
Then, we see that $\tilde{u}_{\infty}(0)=b_{\infty}/\gamma $. Hence, $\tilde{u}_{\infty}\in H^2(0, 1)$ satisfies $-\tilde{u}_{\infty yy}=0$ on $(0, 1)$ with $\tilde{u}_{\infty y}(1)=\tilde{u}_{\infty y}(0)=0$ and $\tilde{u}_{\infty}(1)=0$ so that $\tilde{u}_{\infty} \equiv 0$ on $[0, 1]$. This is a contradiction to $\tilde{u}_{\infty}(0) \neq 0$. Thus, we conclude that Lemma \ref{b-1} holds. 
\end{proof}

Now, we state the result on the large time behavior of a solution as $t \to \infty$. 

\begin{thm}
\label{b-2}
Assume (A1), (A2)' and (A3) and let (P)$(u_0, s_0, b)$ be a solution $(s, u)$ on $[0, \infty)$. 
Then, $s \to \infty$ as $t \to \infty$. 
\end{thm}

We prove this result in the rest of the section. 

%
%
%
%
\subsection{Global estimates}

To prove Theorem \ref{b-2}, we provide some uniform estimates for the solution with respect to time $t$. 
We assume (A1), (A2)' and (A3). Then, by Theorem \ref{t3}, (P)$(u_0, s_0, b)$ has a solution $(s, u) $ on $[0, T]$ for $T>0$ satisfying $0 \leq u\leq b^*/\gamma $ on $[0, s(t)]$ for $t\in [0, T]$. 

\begin{lem}
\label{b-3}
Let $(s, u)$ be a solution of  (P)$(u_0, s_0, b)$ on $[0, \infty)$.
If there exists a constant $C>0$ such that $s(t) \leq C$ for $t>0$, then it holds 
\begin{align}
(i) & \int_0^t |s_t(\tau)|^2 d\tau + \int_0^t |u_z(\tau)|^2_{L^2(0, s(\tau))} d\tau +\int_0^t \left |u(\tau, 0)-\frac{b_{\infty}}{\gamma} \right |^2 d\tau  \leq C_1\mbox{ for }t>0, \label{b-31}\\
(ii) & \int_{0}^t|u_t(\tau)|^2_{L^2(0, s(\tau))}d\tau + |u_z(t)|^2_{L^2(0, s(t))} \leq C_2 \mbox{ for }t>0,  \label{b-32}
\end{align}
where $C_1$ and $C_2$ are positive constants which is independent of time $t$. 
\end{lem}
\begin{proof}
First, we prove that (\ref{b-31}) holds. By (\ref{1-1}) we have that 
\begin{align}
\label{6-1}
\frac{1}{2} \frac{d}{dt} \int_0^{s(t)}\left|u(t)-\frac{b_{\infty}}{\gamma} \right|^2 dz- \frac{s_t(t)}{2} \left |u(t, s(t))-\frac{b_{\infty}}{\gamma} \right|^2 -\int_0^{s(t)} u_{zz}(t) \left( u(t)-\frac{b_{\infty}}{\gamma} \right) dz=0.
\end{align}
For the third term of the left-hand side of (\ref{6-1}), it holds that 
\begin{align}
\label{6-2}
& -\int_0^{s(t)}u_{zz}(t) \left( u(t)-\frac{b_{\infty}}{\gamma} \right)  \nonumber \\
= & -u_z(t, s(t)) \left(u(t, s(t))-\frac{b_{\infty}}{\gamma} \right) + u_z(t, 0) \left( u(t, 0)-\frac{b_{\infty}}{\gamma}  \right) + \int_0^{s(t)}|u_z(t)|^2 dz \nonumber \\
= & s_t(t) \left |u(t, s(t))-\frac{b_{\infty}}{\gamma} \right|^2 + \frac{b_{\infty}}{\gamma} s_t(t) \left (u(t, s(t))-\frac{b_{\infty}}{\gamma} \right) \nonumber \\
& -\beta(b(t)-\gamma u(t, 0))\left (u(0)-\frac{b_{\infty}}{\gamma}\right) + \int_0^{s(t)}|u_z(t)|^2 dz.
\end{align}
By (\ref{6-1}) with (\ref{6-2}) it follows that 
\begin{align}
\label{6-3}
& \frac{1}{2} \frac{d}{dt} \int_0^{s(t)}\left|u(t)-\frac{b_{\infty}}{\gamma} \right|^2 dz + \frac{s_t(t)}{2} \left |u(t, s(t))-\frac{b_{\infty}}{\gamma} \right|^2 + \int_0^{s(t)}|u_z(t)|^2 dz \nonumber \\
& + \frac{b_{\infty}}{\gamma} s_t(t) \left (u(t, s(t))-\frac{b_{\infty}}{\gamma} \right) -\beta(b(t)-\gamma u(t, 0))\left (u(0)-\frac{b_{\infty}}{\gamma}\right)=0.
\end{align} 
Since $s_t(t)=a_0\sigma(u(t, s(t))=a_0u(t, s(t))$, we have that
\begin{align}
\label{6-4}
\frac{b_{\infty}}{\gamma  } s_t(t) \left (u(t, s(t))-\frac{b_{\infty}}{\gamma} \right) = \frac{b_{\infty}}{\gamma} \frac{|s_t(t)|^2}{a_0}-\left(\frac{b_{\infty}}{\gamma}\right)^2 s_t(t). 
\end{align}
Also, it holds that 
\begin{align}
\label{6-5}
& -\beta(b(t)-\gamma u(t, 0))\left (u(0)-\frac{b_{\infty}}{\gamma}\right) \nonumber \\
= & \beta(\gamma u(t, 0)- b_{\infty} + b_{\infty}-b(t)) \left( u(t, 0)-\frac{b_{\infty}}{\gamma}\right) \nonumber \\
= & \beta \gamma \left | u(t, 0)-\frac{b_{\infty}}{\gamma} \right|^2 + \beta(b_{\infty}-b(t)) \left( u(t, 0)-\frac{b_{\infty}}{\gamma}\right) . 
\end{align}
Combining with (\ref{6-3})-(\ref{6-5}), we obtain that 
\begin{align}
\label{6-6}
& \frac{1}{2} \frac{d}{dt} \int_0^{s(t)}\left|u(t)-\frac{b_{\infty}}{\gamma} \right|^2 dz + \frac{s_t(t)}{2} \left |u(t, s(t))-\frac{b_{\infty}}{\gamma} \right|^2 + \int_0^{s(t)}|u_z(t)|^2 dz \nonumber \\
& +\frac{b_{\infty}}{\gamma} \frac{|s_t(t)|^2}{a_0} + \beta \gamma \left | u(t, 0)-\frac{b_{\infty}}{\gamma} \right|^2 \nonumber \\
= & \left(\frac{b_{\infty}}{\gamma}\right)^2 s_t(t) + \beta(b(t)-b_{\infty}) \left( u(t, 0)-\frac{b_{\infty}}{\gamma}\right). 
\end{align}
Here, by the fact that $u(t) \geq 0$ on $[0, s(t)]$ for $t\in [0, T]$ we note that $s_t(t) \geq 0$ for $t\in [0, T]$ and the second term of the left-hand side of (\ref{6-6}) is non-negative. Hence, we derive that 
\begin{align}
\label{6-7}
& \frac{1}{2} \frac{d}{dt} \int_0^{s(t)}\left|u(t)-\frac{b_{\infty}}{\gamma} \right|^2 dz + \int_0^{s(t)}|u_z(t)|^2 dz +\frac{b_{\infty}}{\gamma} \frac{|s_t(t)|^2}{a_0} + \beta \gamma \left | u(t, 0)-\frac{b_{\infty}}{\gamma} \right|^2 \nonumber \\
\leq  & \left(\frac{b_{\infty}}{\gamma}\right)^2 s_t(t) + \beta(b(t)-b_{\infty}) \left( u(t, 0)-\frac{b_{\infty}}{\gamma}\right). 
\end{align}
By using $u(t)\leq b^*/\gamma$ on $[0, s(t)]$ for $t\in [0, T]$ and integrating over $[0, t_1]$ for $t_1\in [0, T]$ we obtain that 
\begin{align}
\label{6-8}
& \frac{1}{2} \int_0^{s(t_1)}\left|u(t_1)-\frac{b_{\infty}}{\gamma } \right|^2 dz + \int_0^{t_1} \int_0^{s(t)}|u_z(t)|^2 dz dt \nonumber \\
& +\frac{b_{\infty}}{a_0\gamma } \int_0^{t_1} |s_t(t)|^2 dt + \beta \gamma \int_0^{t_1} \left | u(t, 0)-\frac{b_{\infty}}{\gamma} \right|^2dt  \nonumber \\
\leq  & \frac{1}{2} \int_0^{s_0}\left|u_0-\frac{b_{\infty}}{\gamma} \right|^2 dz + \left(\frac{b_{\infty}}{\gamma}\right)^2 (s(t_1)-s(0)) + \frac{\beta (b^*+b_{\infty})}{\gamma} \int_0^{t_1}|b_{\infty}-b(t)| dt.
\end{align}
Hence, from $b-b_{\infty}\in L^1(0, \infty)$ in (A2)' and the assumption that $s(t) \leq C$ for $t\in [0, T]$, we conclude that (\ref{b-31}) holds. 

Also, for the estimate (\ref{b-32}), by repeating the proof of Lemma \ref{lem8} we infer that it holds that 
\begin{align}
\label{6-8-1} 
& \int_0^{t_1}\int_0^{s(t)} |u_t(t)|^2 + \frac{1}{2}\int_0^{s(t_1)}|u_z(t_1)|^2dz  \nonumber \\
\leq & \frac{1}{2}\int_0^{s_0}|u_z(0)|^2dz +  |\tilde{u}^3_0(1)| + |\Phi(b(t_1), \tilde{u}(t_1, 0))| + \frac{\beta b^*}{\gamma} \int_0^{t_1}|b_t(t)| dt \mbox{ for }t_1\in [0, T]. 
\end{align}
Therefore, by (\ref{5-21}) and $b_t\in L^1(0, \infty)$ we can find a positive constant $C_2$ which is independent of $t$ such that (\ref{b-32}) holds. This completes the proof of this lemma. 
\end{proof}

%
%
%
%
\subsection{Proof of Theorem \ref{b-2}}
\label{bb}
At the end of the paper, by using the uniform estimate obtained in previous subsection, we complete the proof of Theorem \ref{b-2} concerning the large-time behavior of solutions to (P)$(u_0, s_0, b)$. 

Let us assume (A1), (A2)' and (A3). Then, by Theorem \ref{t3}, we have a solution $(s, u)$ of (P)$(u_0, s_0, b)$ on $[0, T]$ for any $T>0$ such that $0 \leq u\leq b^*/\gamma $ on $[0, s(t)]$ for $t>0$. 

Now, we show Theorem \ref{b-2} by contradiction. Let us assume that there exists a constant $C>0$ such that $s(t)\leq C$ for $t\in [0, T]$. Then, by Lemma \ref{b-3}, we have that 
\begin{align}
& \int_0^t |s_t(\tau)|^2 d\tau + \int_0^t |u_z(\tau)|^2_{L^2(0, s(\tau))} d\tau 
+ \int_0^t \left |u(\tau, 0)-\frac{b_{\infty}}{\gamma } \right |^2 d\tau \leq C_1\mbox{ for }t>0, \label{6-9}\\
& \int_{0}^t|u_t(\tau)|^2_{L^2(0, s(\tau))}d\tau + |u_z(t)|^2_{L^2(0, s(t))} \leq C_2 \mbox{ for }t>0.  \label{6-10}
\end{align}
Here, by (\ref{1-4}), we see that $s_t(t)\geq 0$, and hence $s(t) \geq s_0$ for $t>0$. Also, $u\leq b^*/\gamma $ on $[0, s(t)]$ for $t>0$ so that it holds that $|s_t(t)|\leq a_0b^*/\gamma $ for $t>0$. 
From these results and  the change of variables (\ref{2-0}) we obtain that 
\begin{align}
\label{6-11} 
& \int_0^{t}|\tilde{u}_t(\tau)|^2_{L^2(0, 1)} d\tau \nonumber \\
=&\int_0^t\int_0^{s(\tau)} \frac{1}{s(\tau)} |u_t(\tau, z)+u_z(\tau, z) \frac{z}{s(\tau)}s_t(\tau)|^2dzd\tau \nonumber \\
\leq &\int_0^t\int_0^{s(\tau)} \frac{1}{s(\tau)} |u_t(t, z)|^2dzd\tau + \int_0^t\int_0^{s(\tau)}\frac{2}{s(\tau)}|u_t(\tau, z)||u_z(\tau, z)||s_t(\tau)| dzd\tau \nonumber \\
& + \int_0^t\int_0^{s(\tau)}\frac{1}{s(\tau)}|u_z(\tau, z)|^2|s_t(\tau)|^2 dzd\tau \nonumber \\
\leq & \frac{1}{s_0} \biggl(C_2+\frac{2a_0b^*}{\gamma }C^{1/2}_1C^{1/2}_2 + \left(\frac{a_0b^*}{\gamma }\right)^2 C_1 \biggr) 
\mbox{ for }t>0, 
\end{align}
and 
\begin{align}
\label{6-12}
|\tilde{u}_y(t)|^2_{L^2(0, 1)} = \int_0^{s(t)}\frac{1}{s(t)}|u_z(t) s(t)|^2dzdt \leq C C_2 \mbox{ for }t>0, 
\end{align}
where $C_1$ and $C_2$ are positive constants as in (\ref{6-9}) and (\ref{6-10}). 

Here, for $\{t_n\}$ such that $t_n\to \infty$ as $n\to \infty$, we put $\tilde{u}_n(t, y):=\tilde{u}(t+t_n, y)$ for $(t, y)\in [0, 1]\times [0, 1]$, $s_n(t):=s(t+t_n)$, $b_n(t):=b(t+t_n)$ for $t\in [0, 1]$. By (A2)', it is clear that $b_n \to b_{\infty}$ in $L^1(0, 1)$ as $n\to \infty$. Also, by (\ref{6-9}), (\ref{6-11}) and (\ref{6-12}) we see that $\{s_{nt}\}$ is bounded in $L^2(0, 1)$ and $\{ \tilde{u}_n\} $ is bounded in $W^{1, 2}(0, 1; L^2(0, 1))\cap L^{\infty}(0, 1; H^1(0, 1))$. 
Therefore, we can take a subsequence $\{n_j\}\subset \{n\}$ such that the following convergences holds for some $\tilde{u}_{\infty}\in H^1(0, 1)$ and $s_{\infty} \in \mathbb{R}$ satisfying $s_0\leq s_{\infty}<+\infty$:
$$
\begin{cases}
\tilde{u}_{nj}(0) = \tilde{u}(t_{nj}) \to \tilde{u}_{\infty} \mbox{ in }C([0, 1]), \mbox{ weakly in }H^1(0, 1), \\
s_{nj}(0)=s(t_{nj}) \to s_{\infty} \mbox{ in } \mathbb{R}, \\
\tilde{u}_{njt} \to 0 \mbox{ in } L^2(0, 1; L^2(0, 1)) \mbox{ and } s_{njt} \to 0 \mbox{ in } L^2(0, 1), \\
\tilde{u}_{nj} \to \tilde{u}_{\infty} \mbox{ in }C([0, 1]; L^2(0, 1)), \\
\hspace{1.8cm} \mbox{ weakly in }W^{1, 2}(0, 1;L^2(0, 1)),\\
\hspace{1.8cm} \mbox{ weakly -* in }L^{\infty}(0, 1; H^1(0, 1)), \\
s_{nj} \to s_{\infty} \mbox{ in }C([0, 1]), \mbox{ weakly in }W^{1, 2}(0, 1)
\end{cases}
$$
as $j\to \infty$. Also, by Sobolev's embedding theorem in one dimension (\ref{1d}), it holds that for $y\in [0, 1]$, 
\begin{align}
\label{6-13}
& \int_0^1|\tilde{u}_{nj}(t, y)-\tilde{u}_{\infty}(y)|^2 dt \leq C_e \int_0^1 |\tilde{u}_{nj}(t)-\tilde{u}_{\infty}|_{H^1(0, 1)}|\tilde{u}_{nj}(t)-\tilde{u}_{\infty}|_{L^2(0, 1)} dt. 
\end{align}
Hence, by the strong convergence of $\tilde{u}_{nj}$ and (\ref{6-13}) we see that 
\begin{align}
\label{6-14}
\tilde{u}_{nj}(y) \to \tilde{u}_{\infty}(y) \mbox{ in } L^2(0, 1) \mbox{ at }y=0, 1 \mbox{ as }j\to \infty. 
\end{align}
Now, for each $j$, $(s_{nj}, \tilde{u}_{nj})$ satisfies 
$$
\begin{cases}
\tilde{u}_{njt}(t, y)-\frac{1}{s^2_{nj}(t)}\tilde{u}_{njyy}(t, y)=\frac{ys_{njt}(t)}{s_{nj}(t)}\tilde{u}_{njy}(t, y) \mbox{ for }(t, y)\in Q(1), \\
-\frac{1}{s_{nj}(t)}\tilde{u}_{njy}(t, 0)=\beta(b_{nj}(t)-\gamma \tilde{u}_{nj}(t, 0)) \mbox{ for }t\in(0, 1), \\
 -\frac{1}{s_{nj}(t)}\tilde{u}_{njy}(t, 1)=\tilde{u}_{nj}(t, 1)s_{njt}(t) \mbox{ for }t\in (0, 1),  \\
 s_{njt}(t)=a_0 \tilde{u}_{nj}(t, 1)\mbox{ for }t\in (0, 1). 
\end{cases}
$$  
By letting $j\to \infty$ in the above system and using the strong convergences of $\tilde{u}_{nj}$ and $s_{nj}$, we see that $\tilde{u}_{\infty}\in H^2(0, 1)$ and 
\begin{align}
\label{6-15}
-\displaystyle{\frac{1}{s^2_{\infty}}}\tilde{u}_{\infty yy}=0 \mbox{ on } (0, 1).
\end{align}
Hence, by using the above convergences of $\tilde{u}_{nj}$ and $s_{nj}$, (\ref{6-14}) and (\ref{6-15}) we infer that $(s_{\infty}, \tilde{u}_{\infty})$ satisfies 
\begin{align*}
-\displaystyle{\frac{1}{s_{\infty}}}\tilde{u}_{\infty y}(0)=\beta(b_{\infty}-\gamma \tilde{u}_{\infty}(0)), \quad -\displaystyle{\frac{1}{s_{\infty}}}\tilde{u}_{\infty y}(1)=0, \quad \tilde{u}_{\infty}(1)=0. 
\end{align*} 
Therefore, we see that $(s_{\infty}, \tilde{u}_{\infty})$ is a solution of $(\tilde{\mbox{P}})_{\infty}$ such that $\tilde{u}_{\infty}\in H^2(0, 1)$ and $s_0\leq s_{\infty}<+\infty$. This contradicts that $(\tilde{\mbox{P}})_{\infty}$ does not have a solution (see Lemma \ref{b-1}). Thus, we conclude that $s$ goes to $\infty$ as $t\to \infty$ and Theorem \ref{b-2} holds.

\section{Numerical illustration}
In this section, we use our free boundary model to approximate numerically the  diffusion of a population of solvent molecules (cyclohexane) into a piece of material made of ethylene propylene diene monomer rubber (EPDM). The actual migration experiment and the set of basic parameters are reported in \cite{NMKAMWG}. 

In this framework, we take the effective diffusivity with an order of magnitude higher  and explore briefly of the depth of the penetration front depending on variations in the kinetic parameter $a_0$ arising in (\ref{1-4}). In fact, we look only at a particular instance of the large-time behavior of our problem and point out that, depending on the choice of model parameters, the free boundary position $s(t)$ behaves like a power law of type $t^\beta$, where $\beta$ is typically different than $\frac{1}{2}$ or $1$ as expected for the classical diffusion and for the Case II diffusion, respectively; see \cite{Marat} for a detailed discussion  based on first principles on the large time behavior of sharp diffusion fronts in the transition from glassy to rubbery polymers.

\begin{figure}[h!]
  \centering
  \begin{subfigure}[b]{0.4\linewidth}
    \includegraphics[width=\linewidth]{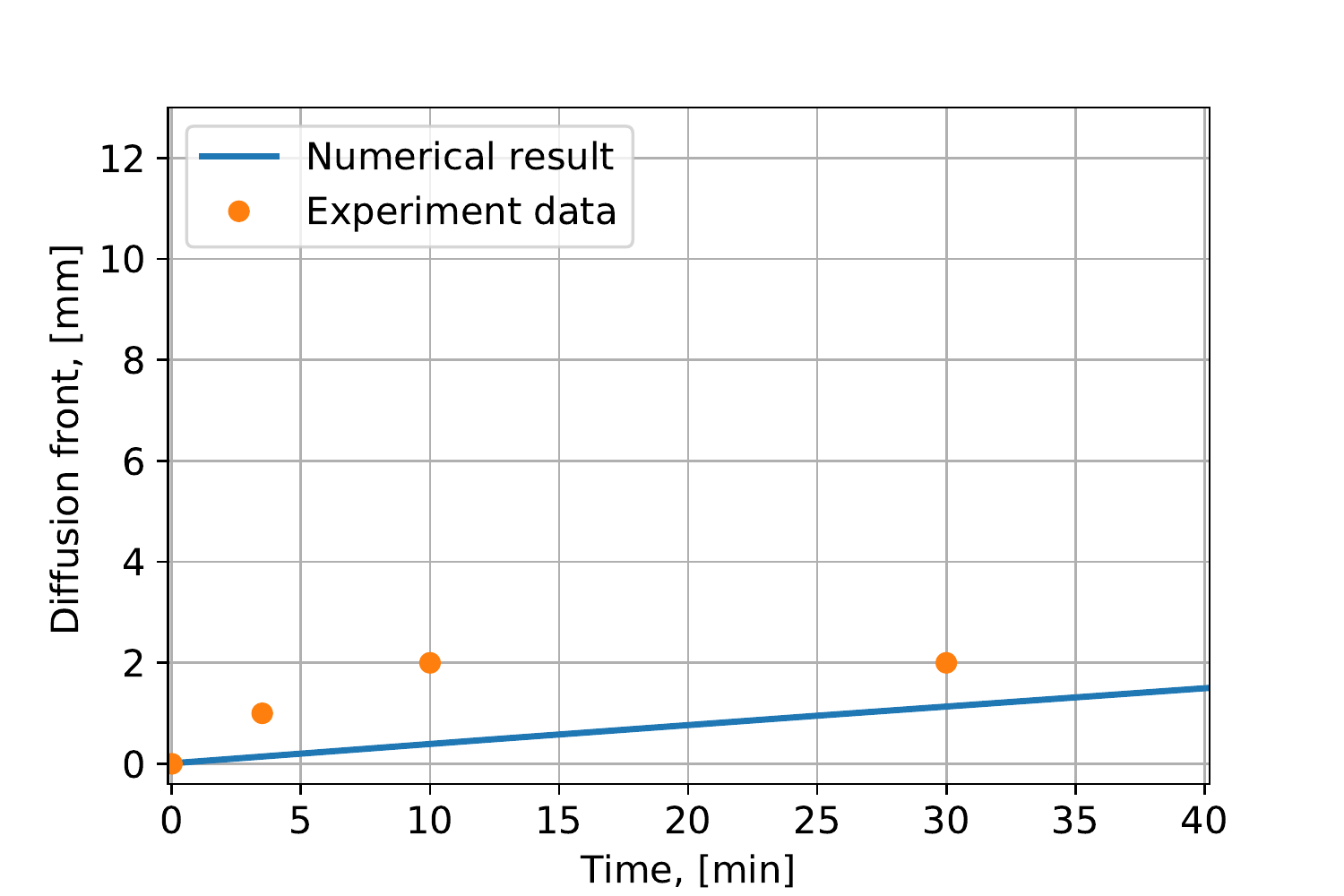}
    \caption{ Comparison to experimental data.}
  \end{subfigure}
  \begin{subfigure}[b]{0.4\linewidth}
   \includegraphics[width=\linewidth]{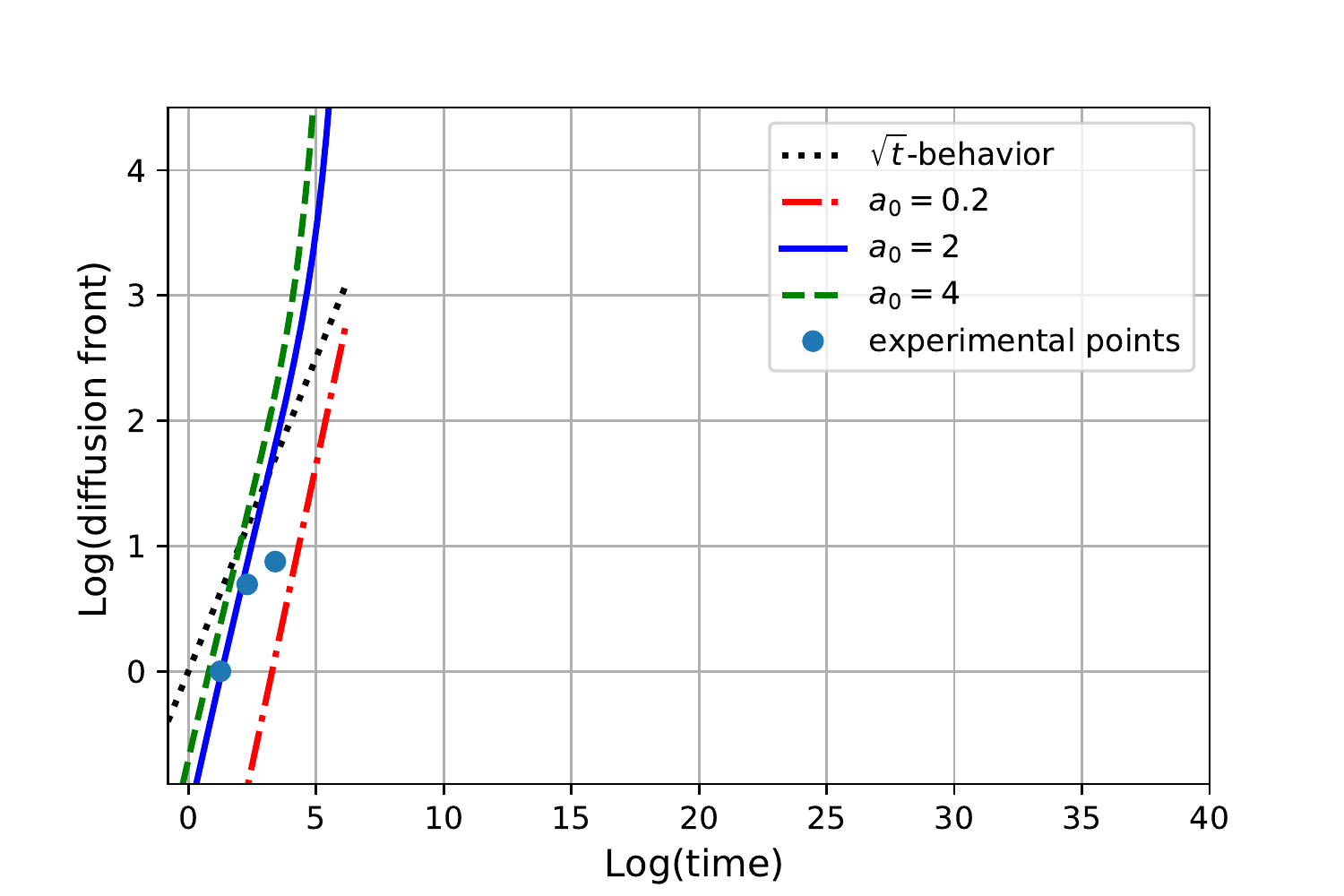}
    \caption{Penetration fronts for various values $a_0$.}
  \end{subfigure}
  \caption{Approximation of the large-time behavior of solutions to $(P)(u_0, s_0, b)$ on $[0,T]$ with $T=5000$ minutes.}
  \label{fig:coffee}
\end{figure}

As shown in Figure \ref{fig:coffee} (a), the behavior of our free boundary seems to be different from the real experimental result.  
From a phenomenological  point of view, a more realistic behavior of the free boundary is obtained in  \cite{NMKAMWG}. 
On the other hand, Theorem \ref{b-2} guarantees  that the growth observed in numerical  results is correct, theoretically. 
Moreover, in order to measure the growth rate for the free boundary, we show numerical results for varying positive constants $a_0$ in Figure \ref{fig:coffee} (b). From these results we conjecture that 
the free boundary position corresponding to Figure \ref{fig:coffee} (a) behaves like $t^{0.41}$. This is a sub-diffusive regime. However, other parameters can bring the front in a super-diffusive regime. Based on our current simulation and mathematical analysis results, we can only state that we expect the free boundary position to follow a power law for large times, but we are, for the moment, unable to  establish rigorously quantitative upper and lower bounds on $s(t)$. Nevertheless, relying also on results from \cite{Souplet}, we hope to be able to adapt some parts of our working technique developed in \cite{AM2} to handle this case. The main difficulty lies  on the fact that it seems that, for a large region in the parameter spaces, our sharp diffusion fronts tend to deviate from $t^\frac{1}{2}$. This makes us wonder what is the most relevant exponent $\beta$ and also for which parameter case and type(s) of rubber-like materials this corresponds.

\section{Discussion}
We were able to prove the global solvability for a one-phase free boundary problem with nonlinear kinetic condition that is meant to describe the migration of diffusants into rubber. Despite of its apparently simple one-dimensional structure, our free boundary model brings in a number of open questions. The most important ones include the identification of an asymptotic dependence  of type $s(t)\sim \mathcal{O}(t^\beta)$  as $t\to+\infty$ and its rigorous mathematical justification. Also, capturing numerically the large time behavior so that a certain power law is preserved requires a special care; compare e.g. the ideas from \cite{Zurek1, Zurek2} to be adapted for the finite element method used here; see \cite{NMKAMWG} for a detailed description of the numerical scheme used in this context.  Of course, to bring the one-dimensional model equations to describe better the physical scenario  of diffusants migrating into rubbers, more modeling components must be added, viz. macroscopic swelling, capillarity transport.  The case of more space dimensions is out of reach as it is not at all clear how the kinetic condition on the moving sharp diffusion front should be formulated especially close to corners or other singularities of the geometry. 

\section*{Acknowledgments} T. A. and A. M. thank the KK Foundation for financial support (project nr. 2019-0213). The work of T. A. is  partially supported also by JSPS KAKENHI Grant Number JP19K03572, while the one by K.K.  is partially supported by JSPS KAKENHI Grant Numbers JP16K17636, JP19K03572 and JP20K03704. 
Fruitful discussions with U. Giese, N. Kr\"oger, R. Meyer (Deutsches Institut f\"ur Kautschuktechnologie, Hannover, Germany) and S. Nepal, Y. Wondmagegne (Karlstad, Sweden) concerning the potential applicability of this research in the case of diffusants migration into polymers have greatly influenced our work.

\medskip
\end{document}